\documentclass[preprint,12pt]{elsarticle}
\usepackage{graphicx}
\usepackage{amssymb}
\usepackage{algcompatible}
\usepackage{algpseudocode}
\usepackage{multirow}
\usepackage{multicol}
\usepackage{xspace}
\usepackage{hyperref}
\usepackage{listings}
\usepackage{url}
\usepackage{csquotes}
\usepackage[mathscr]{euscript} 
\usepackage{amsfonts,epsfig}
\usepackage{enumerate}
\usepackage{amsmath,amscd,amsfonts,amssymb}
\newcommand{\MATLAB}{\textsc{Matlab}\xspace}

\usepackage{lineno,hyperref}
\usepackage{graphicx,subfigure}
\usepackage{amsfonts,epsfig}

\usepackage[mathscr]{euscript} 
\usepackage{enumerate}
\usepackage{amsmath,amscd,amsfonts,amssymb}
\usepackage{times}
\usepackage{extarrows}
\usepackage{stmaryrd}
\usepackage{textcomp}
\usepackage[mathscr]{euscript}
\usepackage{algorithm2e}
\usepackage{colortbl}
\usepackage{color,soul}
\usepackage{newunicodechar}
\usepackage{amsthm}
\theoremstyle{definition}
\newtheorem{thm}{Theorem}

\newtheorem{pre}[thm]{Proposition}
\newtheorem{defn}{Definition}%
\newtheorem{exa}{Example}
\newcommand{\cX}{\underline{\bf X}}
\newcommand{\cR}{\underline{\bf R}}
\newcommand{\cQ}{\underline{\bf Q}}
\newcommand{\cY}{\underline{\bf Y}}

\newcommand{\cU}{\underline{\bf U}}
\newcommand{\cV}{\underline{\bf V}}
\newcommand{\cS}{\underline{\bf S}}

\newcommand{\X}{{\bf X}}

\newcommand{\I}{{\bf I}}
\newcommand{\B}{{\bf B}}

\newcommand{\U}{{\bf U}}

\newcommand{\tS}{{\bf S}}
\newcommand{\V}{{\bf V}}

\newcommand{\Ome}{\underline{\bf \Omega}}

\usepackage{amssymb}

\journal{Journal of Scientific Computing}

\begin{document}

\begin{frontmatter}



\title{A Randomized Algorithm for Tensor Singular Value
Decomposition Using an Arbitrary Number of
Passes}

\author[label1]{Salman Ahmadi-Asl}
\author[label1]{Anh-Huy Phan}
\author[label_2]{Andrzej Cichocki}
\affiliation[label1]{organization={Center for Artificial Intelligence Technology, Skolkovo Institute of Science and Technology,\,Moscow, \,Russia\,s.asl@skoltech.ru},
   }
\affiliation[label_2]{organization={Systems Research Institute of Polish Academy of Science, Warsaw, Poland}
            }





\begin{abstract}
Efficient  and fast computation of a tensor singular value decomposition (t-SVD) with a few passes over the underlying data tensor is crucial because of its many potential applications. The current/existing subspace randomized algorithms need $(2q+2)$ passes over the data tensor to compute a t-SVD, where $q$ is a non-negative integer number (power iteration parameter). In this paper, we propose an efficient and flexible randomized algorithm that can handle any number of passes $q$, which not necessary need be even.
The flexibility of the proposed algorithm in using fewer passes naturally leads to lower computational and communication costs. This advantage makes it particularly appropriate when our task calls for several tensor decompositions or when the data tensors are huge. 
The proposed algorithm is a generalization of the methods developed for matrices to tensors. The expected/average error bound of the proposed algorithm is derived. Extensive numerical experiments on random and real-world data sets are conducted, and the proposed algorithm is compared with some baseline algorithms. The  extensive computer simulation  experiments  demonstrate  that the proposed algorithm is practical, efficient, and in general outperforms  the state of the arts algorithms. 
We also demonstrate how to use  the  proposed method to develop a fast algorithm for the tensor completion problem.
\end{abstract}




\end{frontmatter}


\section{Introduction}
Multidimensional arrays or tensors are natural extensions of matrices and vectors. 
It was in 1927 when the first definition of the tensor rank was presented and the corresponding tensor decomposition termed Canonical Polyadic Decomposition (CPD) was defined \cite{hitchcock1927expression}. Later on, an alternative definition for the tensor rank (Tucker rank) was defined in 1963 \cite{tucker1963implications,tucker1964extension} and a new tensor decomposition called Tucker decomposition was proposed. The Tucker decomposition includes the CPD as a special case. Indeed, there is no unique definition for the tensor rank or the tensor decomposition. Other tensor decompositions include Tensor Train (TT) or Matrix Product State (MPS) \cite{oseledets2011tensor,ostlund1995thermodynamic}, Tensor Chain/Ring (TC) or MPS with periodic boundary conditions \cite{espig2012note,zhao2016tensor}, tensor SVD (t-SVD) \cite{kilmer2011factorization,kilmer2013third} and Hierarchical Tucker decomposition \cite{hackbusch2009new}. 

Tensors have been successfully applied in many machine learning and data analysis tasks such as data reconstruction, data compression, clustering, etc. See \cite{kolda2009tensor,cichocki2017tensor} and the references therein for a comprehensive review of such applications. One of the main challenges in this {work} is developing fast algorithms for the computation of different types of tensor decomposition. For example, to solve the tensor completion problem \cite{song2019tensor}, we usually need to compute tensor decompositions multiple times. When the data tensors are huge or many iterations are necessary for the convergence, these calculations become prohibitively expensive. Therefore, in order to be employed in real-time applications such as traffic data prediction, we need to build fast methods for various types of tensor decompositions. 

The randomization framework has been proven to be an efficient technique, for low-rank matrix computation \cite{halko2011finding} and recently  generalized to the tensors \cite{malik2018low,che2018randomized,malik2021sampling,ahmadi2021randomized}. It is known that randomized algorithms reduce the computational complexity of the deterministic counterparts and also their communication costs. The latter benefit is especially important when the data tensor is very large and stored on several machines. Here, the communication cost is the main concern and we need to access the data tensor as few {times} as possible. In the context of randomization, such methods are called randomized pass-efficient  algorithms \cite{bjarkason2019pass}. For example, in \cite{sun2020low} a pass-efficient randomized algorithm is developed for the Tucker decomposition. The standard randomized subspace iteration method needs $(2q+2)$ passes/views over the data tensor \cite{halko2011finding}, where $q$ is a power iteration parameter. In \cite{bjarkason2019pass}, the author proposed new randomized algorithms for matrices, which do not have this limitation and, for any budget of passes, can compute a low-rank matrix approximation. In this paper, we generalize this idea to the tensor case based on the tubal product (t-product) \cite{kilmer2011factorization,kilmer2013third}. The flexibility of using fewer passes leads to lower communication and computational costs and this is the key benefit of the proposed algorithm in this paper. For example, as will be shown in our experiments, for the image/video compression task, in Section \ref{Sec:Sim}, three passes provided quite good results, while the classical subspace randomized algorithm at least needs 4 passes, and this extra pass can be very expensive for data tensors stored out-of-core or when multiple low tubal rank approximations are required in our task, e.g., tensor completion. Simulations on synthetic and real-world datasets are provided to support the theoretical results. In particular, we present an application of the proposed algorithm for the image/video completion problem. 

Our principal contributions can be summarized as follows:

\begin{itemize}
    \item Developing a pass-efficient randomized algorithm for the computation of the t-SVD with an arbitrary number of passes 
    
    \item Applying the proposed algorithm to reconstruct images/videos with missing pixels
    
    \item Extensive simulation results on synthetic benchmarks and real-world datasets
\end{itemize}

The remainder of this paper is organized as follows: The preliminary concepts and definitions are introduced in Section \ref{Pre}. Section \ref{Sec:TSVD} is devoted to introducing the t-SVD and its computational procedures. The proposed approach is outlined in Section \ref{Sec_pro}. {An application of the proposed algorithm to the tensor completion problem is presented in Section \ref{Sec:APP}.} Simulation results are presented in Section \ref{Sec:Sim} and a conclusion is given in Section \ref{Sec:Conclu}.

\section{Preliminaries}\label{Pre}
Let us introduce the notations and main concepts that we need in the next sections. Tensors, matrices, and vectors are denoted by underlined bold upper case letters e.g., $\underline{\bf X}$, bold upper case letters, e.g., ${\bf X}$, and bold lower case letters, e.g., ${\bf x}$, respectively. 
Slices are produced by fixing all but two modes of a tensor. For a third-order tensor $\underline{\X}$, $\underline{\bf X}(:,:,i),\,\underline{\bf X}(:,i,:)$ and $\underline{\bf X}(i,:,:)$ are called frontal, lateral, and horizontal slices, respective. Fibers are {obtained} by fixing all but one mode. {For a third-order tensor $\underline{\X}$, $\underline{\bf X}(i,j,:)$ is called a tube.} The Frobenius and infinity norms of tensors are denoted by ${\left\|. \right\|_F}$ and $\|.\|_{\infty}$, respectively. The notation ``${\rm conj}$'' denotes the complex conjugate of a complex number or the component-wise complex conjugate of a matrix. The mathematical expectation of a random tensor $\underline{\X}$ is denoted by $\mathbb{E}(\underline{\X})$, and $\lceil n\rceil$ means the nearest integer number greater than or equal to $n$. The element-wise or Hadamard product, is denoted by ``$\oast$''. For a given data tensor $\underline{\X}\in\mathbb{R}^{I_1\times I_2\times \cdots\times I_N}$ and the indicator set $\underline{\bf \Omega}\in\mathbb{R}^{I_1\times I_2\times \cdots\times I_N}$, the projector ${\bf P}_{\underline{\bf \Omega}}$ is defined as follows
\[
{
{\bf P}_{\underline{\bf \Omega}}(\underline{\X})=
\left\{
	\begin{array}{ll}
		\underline{\X}({\bf i})  &  {\bf i}\in\underline{\bf\Omega},\\ 
		0 & {\bf i}\notin\underline{\bf\Omega},
	\end{array}
\right.}
\]
where ${\bf i}=(i_1,i_2,\ldots,i_N)$ is a multi-index and $1\leq i_n\leq I_n,\,\,n=1,2,\ldots,N$. Throughout the paper, we focus only on real and third-order tensors, however, our results can be generalized to complex higher-order tensors straightforwardly. 
 
Before presenting the t-SVD, we first need to present some basic definitions such as the t-product operation and f-diagonal tensors.

\begin{defn} ({t-product})
Let $\underline{\mathbf X}\in\mathbb{R}^{I_1\times I_2\times I_3}$ and $\underline{\mathbf Y}\in\mathbb{R}^{I_2\times I_4\times I_3}$, then the t-product $\underline{\mathbf X}*\underline{\mathbf Y}\in\mathbb{R}^{I_1\times I_4\times I_3}$ is defined as follows
\begin{equation}\label{TPROD}
\underline{\mathbf C} = \underline{\mathbf X} * \underline{\mathbf Y} = {\rm fold}\left( {{\rm circ}\left( \underline{\mathbf X} \right){\rm unfold}\left( \underline{\mathbf Y} \right)} \right),
\end{equation}
where 
\[
{\rm circ} \left(\underline{\mathbf X}\right)
=
\begin{bmatrix}
\underline{\mathbf X}(:,:,1) & \underline{\mathbf X}(:,:,I_3) & \cdots & \underline{\mathbf X}(:,:,2)\\
\underline{\mathbf X}(:,:,2) & \underline{\mathbf X}(:,:,1) & \cdots & \underline{\mathbf X}(:,:,3)\\
 \vdots & \vdots & \ddots &  \vdots \\
 \underline{\mathbf X}(:,:,I_3) & \underline{\mathbf X}(:,:,I_3-1) & \cdots & \underline{\mathbf X}(:,:,1)
\end{bmatrix},
\]
and 
\[
{\rm unfold}(\underline{\mathbf Y})=
\begin{bmatrix}
\underline{\mathbf Y}(:,:,1)\\
\underline{\mathbf Y}(:,:,2)\\
\vdots\\
\underline{\mathbf Y}(:,:,I_3)
\end{bmatrix},\hspace*{.5cm}
\underline{\mathbf Y}={\rm fold} \left({\rm unfold}\left(\underline{\mathbf Y}\right)\right).
\]
\end{defn}

It is known that the t-product operation is indeed the circular convolution operator and can be computed using the Fast Fourier Transform (FFT). To this end, the FTT operator is applied to all tubes of two tensors $\underline{\bf X},\,\underline{\bf Y}$, and obtains new tensors $\widehat{\underline{\bf X}},\,\widehat{\underline{\bf Y}}$. Then, we multiply the frontal slices of the tensors $\widehat{\underline{\bf X}},\,\widehat{\underline{\bf Y}}$ to get the new tensor $\widehat{\underline{\bf Z}}$. Finally, we apply the Inverse FFT (IFFT) operator to all tubes of the tensor $\widehat{\underline{\bf Z}}$. This procedure is summarized in Algorithm \ref{ALG:TSVDP}, which is the optimized version of the t-product, because it needs only the FFT of the $\lceil \frac{I_3+1}{2}\rceil$ first frontal slices suggested by the works \cite{hao2013facial,lu2019tensor}, while the original papers (\cite{kilmer2011factorization,kilmer2013third}) consider the FFT of all frontal slices. Note that ${\rm fft}\left( {\underline{\mathbf Z},[],3} \right)$ is equivalent to computing the FFT of all tubes of the tensor $\underline{\bf Z}$.
The t-product can be defined according to an arbitrary invertible transform \cite{kernfeld2015tensor}. For example, the unitary transform matrices were utilized in  \cite{song2020robust}, instead of discrete Fourier transform matrices. It was also proposed in \cite{jiang2020framelet} to use  non-invertible transforms instead of unitary matrices.

It can be proven that for a tensor $\X\in\mathbb{R}^{I_1\times I_2\times I_3}$, we have 
\begin{eqnarray}\label{eq_fou}
\|\underline{\X}\|^2_F=\frac{1}{I_3}\sum_{i=1}^{I_3}\|\widehat{\underline{\X}}(:,:,i)\|_F^2,
\end{eqnarray}
where $\widehat{\underline{\X}}(:,:,i)$ is the $i$-th frontal slice of the tensor $\widehat{\underline{\X}}={\rm fft}(\underline{\X},[],3)$, see \cite{lu2019tensor,zhang2018randomized}.

\RestyleAlgo{ruled}
\LinesNumbered
\begin{algorithm}
{\small
\SetKwInOut{Input}{Input}
\SetKwInOut{Output}{Output}\Input{Two data tensors $\underline{\mathbf X} \in {\mathbb{R}^{{I_1} \times {I_2} \times {I_3}}},\,\,\underline{\mathbf Y} \in {\mathbb{R}^{{I_2} \times {I_4} \times {I_3}}}$} 
\Output{t-product $\underline{\mathbf C} = \underline{\mathbf X} * \underline{\mathbf Y}\in\mathbb{R}^{I_1\times I_4\times I_3}$}
\caption{t-product in the Fourier domain \cite{kilmer2011factorization}}\label{ALG:TSVDP}
      {
      $\widehat{\underline{\mathbf X}} = {\rm fft}\left( {\underline{\mathbf X},[],3} \right)$;\\
      $\widehat{\underline{\mathbf Y}} = {\rm fft}\left( {\underline{\mathbf Y},[],3} \right)$;\\
\For{$i=1,2,\ldots,\lceil \frac{I_3+1}{2}\rceil$}
{                        
$\widehat{\underline{\mathbf C}}\left( {:,:,i} \right) = \widehat{\underline{\mathbf X}}\left( {:,:,i} \right)\,\widehat{\underline{\mathbf Y}}\left( {:,:,i} \right)$;\\
}
\For{$i=\lceil\frac{I_3+1}{2}\rceil+1\ldots,I_3$}{
$\widehat{\underline{\mathbf C}}\left( {:,:,i} \right)={\rm conj}(\widehat{\underline{\mathbf C}}\left( {:,:,I_3-i+2} \right))$;
}
$\underline{\mathbf C} = {\rm ifft}\left( {\widehat{\underline{\mathbf C}},[],3} \right)$;   
       	} 
        }
\end{algorithm}

\begin{defn} ({Transpose})
Let $\underline{\mathbf X}\in\mathbb{R}^{I_1\times I_2\times I_3}$ be a given tensor. Then the transpose of the tensor $\underline{\mathbf X}$ is denoted by $\underline{\mathbf X}^{T}\in\mathbb{R}^{I_2\times I_1\times I_3}$, which is constructed by transposing all its frontal slices and then reversing the order of transposed frontal
slices $2$ through $I_3$.
\end{defn}

\begin{defn} ({Identity tensor})
The tensor $\underline{\mathbf I}\in\mathbb{R}^{I_1\times I_1\times I_3}$ is called identity if its first frontal slice is an identity matrix of size $I_1\times I_1$ and all other frontal slices are zero. It is easy to show $\underline{\I}*\underline{\X}={\underline \X}$ and $\underline{\X}*\underline{\I} =\underline{\X}$ for all tensors of conforming sizes.
\end{defn}
\begin{defn} ({Orthogonal tensor})
A tensor $\underline{\mathbf X}\in\mathbb{R}^{I_1\times I_1\times I_3}$ is orthogonal if ${\underline{\mathbf X}^T} * \underline{\mathbf X} = \underline{\mathbf X} * {\underline{\mathbf X}^ T} = \underline{\mathbf I}$.
\end{defn}

\begin{defn} ({f-diagonal tensor})
If all frontal slices of a tensor are diagonal then the tensor is called an f-diagonal tensor.
\end{defn}

\begin{defn}
(Random tensor) A tensor $\underline{\bf \Omega}$ is random if its first frontal slice $\underline{\bf \Omega}(:,:,1)$ has independent and identically distributed (i.i.d) elements, while the other frontal slices are zero.
\end{defn}

The \MATLAB implementation of many operations in the {t-product} format can be found in the following useful toolbox:\\ \url{https://github.com/canyilu/Tensor-tensor-product-toolbox}.

\section{Tensor SVD and Tensor QR decomposition}\label{Sec:TSVD}
Tensor SVD (t-SVD) was originally proposed in (\cite{kilmer2011factorization,kilmer2013third,braman2010third,gleich2013power}) to represent a third-order tensor as a product of three third-order tensors, where all frontal slices of the middle tensor are diagonal, (Figure \ref{Pic7}, provides a graphical illustration on the t-SVD {and} its truncated version). For a generalization of the t-SVD to higher-order {tensors,} see \cite{martin2013order}. This decomposition has found interesting applications such as completion, clustering, compression, and background initialization in video surveillance, see \cite{sobral2017matrix,jiang2017novel,jiang2019robust,he2020robust}. 

The tubal rank is defined as the number of nonzero tubes of the middle tensor. It is interesting to note that, unlike other tensor decompositions, the truncated t-SVD provides the best approximation in the least-squares sense for any unitary invariant tensor norm \cite{kilmer2011factorization}. 

\begin{figure}
\begin{center}
\includegraphics[width=9 cm,height=4.5 cm]{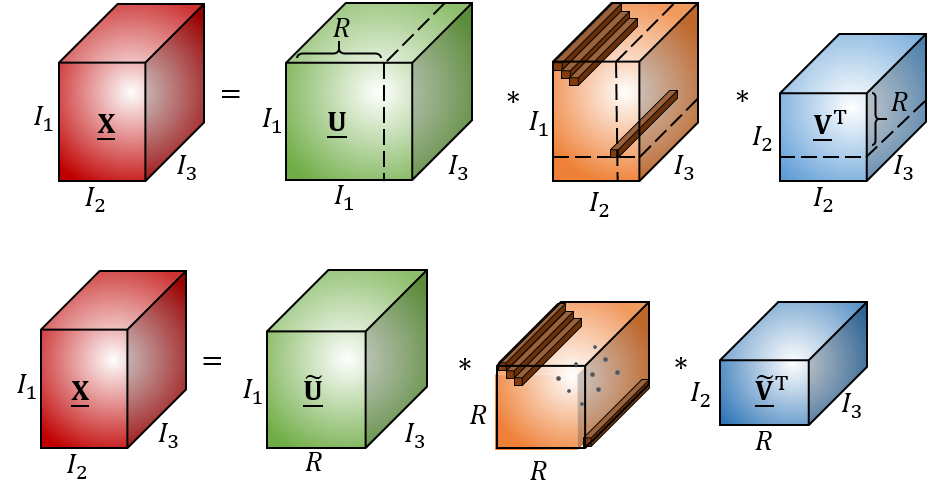}\\
\caption{\small{Illustration of ({\bf a}) Tensor SVD (t-SVD) and ({\bf b}) truncated t-SVD for a third-order tensor 
\cite{ahmadi2021cross}.}}\label{Pic7}
\end{center}
\end{figure}

Let $\underline{\bf X}\in\mathbb{R}^{I_1\times I_2\times I_3}$, then the t-SVD of the tensor $\underline{\bf X},$ admits the model $\underline{\bf X} = \underline{\bf U} * \underline{\bf S}* {\underline{\bf V}^T},$ where $\underline{\bf U}\in\mathbb{R}^{I_1\times R\times I_3}$, $\underline{\bf V}\in\mathbb{R}^{I_2\times R\times I_3}$ are orthogonal tensors and the tensor $\underline{\bf S}\in\mathbb{R}^{R\times R\times I_3}$ is f-diagonal (\cite{kilmer2011factorization,kilmer2013third}). Computation of the t-SVD is performed in the Fourier domain and this is summarized in Algorithm \ref{ALG:t-SVD}. {Algorithm \ref{ALG:t-SVD} is a faster version of the original truncated t-SVD \cite{kilmer2013third,kilmer2011factorization} and was developed in \cite{kernfeld2015tensor,lu2019tensor}.} It requires only the SVD of the $\lceil\frac{I_3+1}{2}\rceil$ first frontal slices \cite{lu2019tensor}.
The original t-SVD \cite{kilmer2011factorization,kilmer2013third} involves the SVD of all frontal slices in the Fourier domain.
Similar to the t-product, instead of the discrete Fourier transform matrices, one can define the t-SVD according to an arbitrary unitary transform matrix. It is shown in \cite{jiang2020framelet} that this can provide a t-SVD with a lower tubal rank.

\RestyleAlgo{ruled}
\LinesNumbered
\begin{algorithm}
{\smaller
\SetKwInOut{Input}{Input}
\SetKwInOut{Output}{Output}\Input{A data tensor $\underline{\mathbf X} \in {\mathbb{R}^{{I_1} \times {I_2} \times {I_3}}}$ and a tubal rank $R$;} 
\Output{${\underline{\mathbf U}_R} \in {\mathbb{R}^{{I_1} \times R \times {I_3}}},\,\,{\underline{\mathbf S}_R} \in {\mathbb{R}^{R \times R \times {I_3}}},$ ${\underline{\mathbf V}_R} \in {\mathbb{R}^{{I_2} \times R \times {I_3}}}$;}
\caption{Truncated t-SVD \cite{kilmer2011factorization,kilmer2013third}}\label{ALG:t-SVD}
      {
$\widehat{\underline{\mathbf X}}= {\rm fft}\left( {\underline{\mathbf X},[],3} \right)$;\\
\For{$i=1,2,\ldots,\lceil\frac{I_3+1}{2}\rceil$}
{                        
$[{\mathbf U},{\mathbf S},{\mathbf V}] = {\rm Truncated\operatorname{-}SVD}\left( {\widehat{\underline{\mathbf X}}(:,:,i),R} \right)$;\\
${\widehat{\underline{\mathbf U}}}\left( {:,:,i} \right) = {{\mathbf U}}$;\\
${\widehat{\underline{\mathbf S}}}\left( {:,:,i} \right) = {{\mathbf S}}$;\\ 
${\widehat{\underline{\mathbf V}}}\left( {:,:,i} \right) = {{\mathbf V}}$;\\
}
\For{$i=\lceil\frac{I_3+1}{2}\rceil+1,\ldots,I_3$}{
$\widehat{\U}(:,:,i)={\rm conj}(\widehat{\cU}(:,:,I_3-i+2))$;\\
$\widehat{\tS}(:,:,i)=\widehat{\underline{\tS}}(:,:,I_3-i+2)$;\\
$\widehat{\V}(:,:,i)={\rm conj}(\widehat{\cV}(:,:,I_3-i+2))$;}
${\underline{\mathbf U}}_R = {\rm ifft}\left( {{\widehat{\underline{\mathbf U}}},[],3} \right),\,{\underline{\mathbf S}}_R = {\rm ifft}\left( {{\widehat{\underline{\mathbf S}}},[],3} \right),\,{\underline{\mathbf V}}_R = {\rm ifft}\left( {\widehat{\underline{\mathbf V}},[],3} \right);$  
       	} }      	
\end{algorithm}

The tensor QR (t-QR) decomposition can be defined similarly based on the t-product. To be precise, for the t-QR decomposition of a tensor $\underline{\mathbf X}\in\mathbb{R}^{I_1\times I_2\times I_3},\,$ i.e., $\underline{\mathbf X} = \underline{\mathbf Q} * \underline{\mathbf R}$, we first compute the FFT of the tensor $\underline{\mathbf X}$ as 
\begin{equation}\label{fftten}
\underline{\widehat{\mathbf X}}={\rm fft}({\underline{\mathbf X}},[],3),
\end{equation}
and then the QR decompositions of all frontal slices of the tensor $\underline{\widehat{\mathbf X}}$ are computed as follows
\begin{equation}\label{LU_tubal}
\underline{\widehat{\mathbf X}}(:,:,i)=\underline{\widehat{\mathbf Q}}(:,:,i)\,\,\underline{\widehat{\mathbf R}}(:,:,i).
\end{equation}
Finally, the IFFT operator is applied to the tensors $\underline{\widehat{\mathbf Q}}$ and $\underline{\widehat{\mathbf R}}$ to compute the tensors $\underline{{\mathbf Q}}$ and $\underline{{\mathbf R}}$. A given data tensor $\underline{\bf X}$, can be orthogonalized by applying the t-QR decomposition and taking the $\underline{\bf Q}$ part. We use the notation ${\rm orth}(\underline{\bf X})$ to denote this operation. It is not difficult to see that the t-SVD and the t-QR decomposition for matrices ($I_3=1$) are reduced to the classical matrix SVD and QR decomposition. Here, ${\rm orth}({\bf X})$ gives the orthogonal part of the matrix QR decomposition. In Matlab, this is equivalent to ${\bf Q}={\bf qr}({\bf X},0)$. The tubal LU (t-LU) decomposition  \cite{zhu2022tensor} can be defined in an analogous way by replacing LU decomposition with the QR decomposition in \eqref{LU_tubal}. {We remark that the rank-revealing QR and LU decompositions and their randomized variants were extended to tensors based on the t-product in \cite{zhu2022tensor}} and can also be used in this work. 

\section{A Pass-Efficient randomized algorithm for computation of the t-SVD}\label{Sec_pro}

It is obvious that Algorithm \ref{ALG:t-SVD} is prohibitive for large-scale tensors because of the computation of multiple SVDs of some large-scale matrices. Let us first describe the idea of the random projection technique  \cite{zhang2018randomized,qi2021t} to tackle this problem. {In the first stage of the random projection method, the size of a given data tensor $\underline{\bf X}$ is reduced by multiplying the $\underline{\bf X}$ with a random tensor $\underline{\bf\Omega} \in {\mathbb{R}^{{I_2} \times {(R+P)} \times {I_3}}}$ as $\underline{\bf Y} = \underline{\bf X}*\underline{\bf \Omega}$. Then, the tensor $\underline{\bf Y}$ is orthogonalized through the t-QR decomposition of the tensor $\underline{\bf Y}$ to provide the following low tubal rank approximation:
\begin{equation}\label{ltrank}
\underline{\bf X} \cong \underline{\bf Q}*\underline{\bf B},
\end{equation}
where $\underline{\bf B} = {\underline{\bf Q}^T}*\underline{\bf X},$ and $\underline{\bf Q} \in {\mathbb{R}^{{I_1} \times R \times {I_3}}},\,\,\underline{\bf B} \in {\mathbb{R}^{R \times {I_2} \times {I_3}}}$. Note $R+P\ll{I_2}$, where $R$ is an estimation of the tubal rank and $P$ is the {\it oversampling} parameter to better capture the action of the tensor $\underline{\bf X}$ \cite{zhang2018randomized}. From the low tubal rank approximation \eqref{ltrank}, and the t-SVD of {$\underline{\B}$} as $\underline{\bf B} = \widehat{\underline{\bf U}}*\underline{\bf S}*{\underline{\bf V}^T},$ we can recover the t-SVD of $\underline{\bf X}$ as $\underline{\bf X} = \left( {\underline{\bf Q}*\widehat{\underline{\bf U}}} \right) * \underline{\bf S} * {\underline{\bf V}^T}.$ It is worth mentioning that the tensor $\B$ is smaller than the original data tensor $\X$ and requires less memory and computational resources.} This approach is efficient when the singular values of the frontal slices decay very fast; otherwise, the power iteration should be utilized. Here, the original tensor $\underline{\bf X}$ is replaced with $\underline{\bf Z}=\left(\underline{\bf X}*\underline{\bf X}^T\right)^q*\underline{\bf X}$ and the above-mentioned procedure is applied to $\underline{\bf Z}$. Due to stability issues, the tensor $\underline{\bf Z}$ should not be explicitly computed. Instead, it can be efficiently computed using the subspace iteration method \cite{halko2011finding,zhang2018randomized}.

The basic randomized algorithm for the low-rank matrix approximation is summarized in Algorithm \ref{ALg_1}. The "for" loop (Lines $3\operatorname{-}6$), is the power iteration technique and is used when the singular values of a matrix do not decay fast. It has been established that in practice, $q=1,2,3$ are sufficient to achieve good accuracy. In addition, $P$ is the oversampling parameter, which helps to better capture the range of the matrix. These strategies have been generalized to develop fast algorithms for computation of the t-SVD.
Algorithm \ref{ALgRRM}, is an extension of Algorithm \ref{ALg_1} to the case of tensors based on the t-product. For computational efficiency, one can replace the t-QR decomposition in Algorithm \ref{ALgRRM} with the t-LU decomposition \cite{zhu2022tensor,li2017algorithm,erichson2016randomized}, or alternatively use a combination of t-QR and t-LU decompositions \cite{li2017algorithm,voronin2015rsvdpack}. Similar to Algorithm \ref{ALg_1}, the power iteration procedure discussed earlier is performed in Algorithm \ref{ALgRRM}, in Lines ($3\operatorname{-}6$). For power iteration $q$, Algorithms \ref{ALg_1} and \ref{ALgRRM} need to pass the data tensor $(2q+2)$ times. Indeed, for Algorithm \ref{ALgRRM} two passes in lines $(2,\& 7)$ and $2q$ passes in lines $(3\operatorname{-}6)$ are required. This means that the number of passes over the data tensor is always an even number. To the best of our knowledge, the only paper, which proposes a single-pass randomized algorithm for computing the t-SVD is \cite{qi2021t}, which is a generalization of those proposed in \cite{battaglino2018practical} from matrices to tensors. In \cite{bjarkason2019pass}, the authors resolved the drawback of the mentioned limitation of the randomized subspace algorithms for matrices and developed randomized algorithms, which are applicable for a budget of any number of passes. This algorithm is presented in Algorithm \ref{ALg_2}. Motivated by this efficient algorithm, we extend it to the tensor case based on the t-product. 
The proposed pass-efficient algorithm can compute a low tubal rank approximation of the underlying data tensor using an arbitrary number of passes and is presented in Algorithm \ref{ALg_PASEFF}. Besides, we exploit the proposed method to develop a fast algorithm to solve the tensor completion problem as a practical application.

Now, we describe the main procedure for computing the t-SVD in a pass-efficient way. Assuming that we afford a budget of $v$ passes (the number of possible passes) for computing the t-SVD, it is clear that if $v\geq 4$ is even, then we can use the classical randomized subspace Algorithm \ref{ALgRRM} with the power iteration parameter $q=(v-2)/2$. So, how about making the amount of passes odd?

Inspired by the idea presented in \cite{bjarkason2019pass}, we suggest the following procedures for computation of the t-SVD given $v\geq 2$ passes:
\begin{itemize}
    \item (If $v$ is even) {Construct} an orthonormal tensor $\cQ$ for the range of the tensor $(\cX*\cX^T)^{(v-2)/2}*\cX*\Ome$. Then, compute the truncated t-SVD of the tensor $\cX^T*\cQ$ as
    \[
    [\cV,\cS,\widehat{\cU}]={\rm Trunacted \,\,t}\operatorname{-}{\rm SVD}(\cX^T*\cQ)
    \]
    and set $\cU=\cQ*\widehat{\cU}$.
    \\
    
    \item (If $v$ is odd) {Construct} an orthonormal tensor $\cQ$ for the range of the tensor $(\cX^T*\cX)^{(v-1)/2}*\Ome$. Then, compute the truncated t-SVD of the tensor $\cX*\cQ$ as  
    $$[\cU,\cS,\widehat{\cV}]={\rm Trunacted\,\,t}\operatorname{-}{\rm SVD}(\cX*\cQ)$$
     and set $\cV=\cQ*\widehat{\cV}$.\\
\end{itemize}

\RestyleAlgo{ruled}
\begin{algorithm}
{\smaller
\LinesNumbered
\SetKwInOut{Input}{Input}
\SetKwInOut{Output}{Output}  \Input{A data matrix ${\bf X}\in\mathbb{R}^{I_1\times I_2}$; a matrix rank $R$; Oversampling $P$ and the power iteration $q$.}  \Output{Truncated SVD: ${\bf X}\cong {\bf U}{\bf S}{\bf V}^T$}
\caption{Classical randomized subspace method for computation of the Truncated SVD \cite{voronin2015rsvdpack}}\label{ALg_1}
${\bf \Omega}={\rm randn}(I_2,P+R)$;\\
$[{\bf Q}^{(1)},\sim] = {{\rm QR}}({\bf X}{\bf \Omega})$;\\
\For{$i=1,2,\ldots,q$}{
$[{\bf Q}^{(2)},\sim] = {{\rm QR}}({\bf X}^T{\bf Q}^{(1)})$;\\
$[{\bf Q}^{(1)},\sim] = {{\rm QR}}({\bf X}{\bf Q}^{(2)})$;\\
}
$[{\bf Q}^{(2)},{\bf R}] = {{\rm QR}}({\bf X}^T{\bf Q}^{(1)})$;\\
$[\widehat{\bf V},{\bf S},\widehat{\bf U}] = {\rm Truncated }$ $\,{\rm SVD}({\bf R},R)$;\\
${\bf V}={\bf Q}^{(2)}\widehat{\bf V}$;\\
${\bf U}={\bf Q}^{(1)}\widehat{\bf U}$;
}
\end{algorithm}

\RestyleAlgo{ruled}
\begin{algorithm}
{\smaller
\LinesNumbered
\SetKwInOut{Input}{Input}
\SetKwInOut{Output}{Output} \Input{A data tensor $\underline{\bf X}\in\mathbb{R}^{I_1\times I_2\times I_3}$; a tubal rank $R$; Oversampling $P$ and the power iteration $q$.}  \Output{Truncated t-SVD: $\cX\cong \cU*\underline{\tS}*\cV^T$}
\caption{Classical randomized subspace method for computation of the Truncated t-SVD \cite{zhang2018randomized}}\label{ALgRRM}
$\Ome={\rm randn}(I_2,P+R,I_3)$;\\
$[\cQ^{(1)},\sim] = {\rm t-QR}(\cX*\Ome)$;\\
\For{$i=1,2,\ldots,q$}{
$[\cQ^{(2)},\sim] = {{\rm t-QR}}(\cX^T*\cQ^{(1)})$;\\
$[\cQ^{(1)},\sim] = {{\rm t-QR}}(\cX*\cQ^{(2)})$;\\
}
$[\cQ^{(2)},\cR] = {{\rm t-QR}}(\cX^T*\cQ^{(1)})$;\\
$[\widehat{\cV},\cS,\widehat{\cU}] = {\rm Truncated }$ ${\rm t}$-${\rm SVD}(\cR,R)$;\\
$\cV=\cQ^{(2)}*\widehat{\cV}$;\\
$\cU=\cQ^{(1)}*\widehat{\cU}$;
}
\end{algorithm} 

\RestyleAlgo{ruled}
\begin{algorithm}
{\smaller
\LinesNumbered
\SetKwInOut{Input}{Input}
\SetKwInOut{Output}{Output}  \Input{A data matrix ${\bf X}\in\mathbb{R}^{I_1\times I_2}$; a matrix rank $R$; Oversampling $P$ and the budget of number of passes $v$.} \Output{Truncated SVD: ${\bf X}\cong {\bf U}{\bf S}{\bf V}^T$}
\caption{Randomized truncated SVD with an arbitrary number of passes \cite{bjarkason2019pass}}\label{ALg_2}
${\bf Q}^{(1)}={\rm randn}(I_2,P+R)$;\\
\For{$i=1,2,\ldots,v$}{
\eIf{$i$ is odd}
{$[{\bf Q}^{(2)},{\bf R}^{(2)}] = {{\rm QR}}({\bf X}{\bf Q}^{(1)})$;}
{$[{\bf Q}^{(1)},{\bf R}^{(1)}] = {{\rm QR}}({\bf X}^T{\bf Q}^{(2)})$;}
}
\eIf{$v$ is even}
{$[\widehat{\bf V},{\bf S},\widehat{\bf U}]=$Truncated SVD$({\bf R}^{(1)},R)$;}
{$[\widehat{\bf U},{\bf S},\widehat{\bf V}]=$Truncated SVD$({\bf R}^{(2)},R)$;}
${\bf V}={\bf Q}^{(1)}\widehat{\bf V}$;\\
${\bf U}={\bf Q}^{(2)}\widehat{\bf U}$;
}
\end{algorithm} 

\RestyleAlgo{ruled}
\begin{algorithm}
{\smaller
\LinesNumbered
\SetKwInOut{Input}{Input}
\SetKwInOut{Output}{Output}\Input{A data tensor $\underline{\bf X}\in\mathbb{R}^{I_1\times I_2\times I_3}$; a tubal rank $R$; Oversampling $P$ and the budget of number of passes $v$.}  \Output{Truncated t-SVD: $\cX\cong \cU*\underline{\tS}*\cV^T$}
\caption{Proposed pass-efficient randomized truncated t-SVD with an arbitrary number of passes}\label{ALg_PASEFF}
$\cQ^{(1)}={\rm randn}(I_2,P+R,I_3)$;\\
\For{$i=1,2,\ldots,v$}{
\eIf{$i$ is odd}
{$[\cQ^{(2)},\cR^{(2)}] = {{\rm t-QR}}(\cX*\cQ^{(1)})$;}
{$[\cQ^{(1)},\cR^{(1)}] = {{\rm t-QR}}(\cX^T*\cQ^{(2)})$;}
}
\eIf{$v$ is even}
{$[\widehat{\cV},\cS,\widehat{\cU}]=$Truncate t-SVD$(\cR^{(1)},R)$;}
{$[\widehat{\cU},\cS,\widehat{\cV}]=$Truncated t-SVD$(\cR^{(2)},R)$;}
$\cV=\cQ^{(1)}*\widehat{\cV}$;\\
$\cU=\cQ^{(2)}*\widehat{\cU}$;
}
\end{algorithm} 

{Note that the proposed algorithm \ref{ALg_PASEFF} requires an estimation of the tubal rank as input but a fixed-precision algorithm is proposed in \cite{ahmadi2022efficient} that for a given approximation bound gives the tubal rank and corresponding low tubal rank approximation. The proposed randomized algorithm offers more flexibility in passing the underlying data tensor.} The next theorem gives the expected/average upper bound of the approximated tensor computed by Algorithm \ref{ALg_1} and we use it to derive similar results for Algorithm \ref{ALg_2}. 

\begin{thm}\label{THMM_1}
\cite{zhang2018randomized}
  (Average Frobenius error for Algorithm \ref{ALg_1}). Let ${\bf X}\in\mathbb{R}^{I_1\times I_2}$ and ${\bf \Omega}\in\mathbb{R}^{I_2\times (R+P)}$ be a given matrix and a 
Gaussian random matrix, respectively with the oversampling parameter $P \geq 2$. Suppose ${\bf Q}$ is obtained from Algorithm \ref{ALg_1}, then
\[
\mathbb{E}\left(\|{\bf X}-{\bf Q}{\bf Q}^T{\bf X}\|_F^2\right)\leq\left(1+\frac{R}{P-1}\tau_R^{4q}\right)\left(\sum_{j=R+1}^{\min\{I_1,I_2\}}\sigma^2_j\right),
\]
where $R$ is a matrix rank, $q$ is the power iteration, $\sigma_j$ is the $j$-th singular value of ${\bf X}$, and $\tau_R={\sigma_{R+1}}/{\sigma_R}\ll 1$ is the singular value gap.
\end{thm}

It is not difficult to check that for an even number of passes, e.g., $v=2q+2$, Algorithm \ref{ALg_2} is equivalent to Algorithm \ref{ALg_1}, so Theorem \ref{THMM_1} can be used for its error analysis. However, given an odd number of passes, Algorithm 5's error analysis is presented in Theorem \ref{thm_2}. Using this, we can derive the average Frobenius norm error for Algorithm \ref{ALg_PASEFF}. Note that in \cite{bjarkason2019pass}, some average error bounds (in the spectral norm) have been established for Algorithm \ref{ALg_2}, but here we consider the Frobenius norm because it facilitates the derivation of the average error bound of the approximations achieved by the proposed Algorithm \ref{ALg_PASEFF}.

\begin{thm}\label{thm_2}
 (Average Frobenius error for Algorithm \ref{ALg_2}). Let ${\bf X}\in\mathbb{R}^{I_1\times I_2}$ and ${\bf Q}^{(1)}\in\mathbb{R}^{I_2\times (R+P)}$ be a
given matrix and a Gaussian random matrix respectively with $P \geq 2$ being the oversampling parameter. Suppose ${\bf Q}$ is obtained from Algorithm \ref{ALg_2}, with an odd number of passes $v$, then
\[
\mathbb{E}\left(\|{\bf X}-{\bf X}{\bf Q}{\bf Q}^T\|_F^2\right)\leq\left(1+\frac{R}{P-1}\tau_R^{2(2v-1)}\right)\left(\sum_{j=R+1}^{\min\{I_1,I_2\}}\sigma^2_j\right),
\]
where $R$ is the matrix rank, $\sigma_j$ is the $j$-th singular value of ${\bf X}$, and $\tau_R={\sigma_{R+1}}/{\sigma_R}\ll 1$ is the singular value gap.
\end{thm}

\begin{proof}
See the Appendix.
\end{proof}

Theorem \ref{thm_3} provides the average error bound for Algorithm \ref{ALgRRM}. Similar to the matrix case, Algorithm \ref{ALg_PASEFF} is reduced to Algorithm \ref{ALgRRM} when an even number of passes is used hence Theorem \ref{thm_3} can be used for its error analysis. For the case of an odd number of passes, we present Theorem \ref{thm_4}, and its proof is quite similar to the proof of Theorem \ref{thm_3}.

\begin{thm}\label{thm_3}
\cite{zhang2018randomized}
 Given an $I_1\times I_2\times I_3$ tensor $\underline{\bf X}$ and a Gaussian tensor   $\underline{\bf \Omega}$ of size $I_2\times (R+P) \times I_3$, if $\underline{\bf Q}$ is obtained
from Algorithm \ref{ALgRRM}, then
\begin{eqnarray*}
\mathbb{E}\left(\|{\cX}-{\cQ}*{\cQ}^T*{\cX}\|_F\right)\leq \left(\frac{1}{I_3}\sum_{i=1}^{I_3}\left(1+\frac{R}{P-1}(\widetilde{\tau}_R^{(i)})^{4q}\right)\right.\\
\left.
\left(\sum_{j>R}(\widetilde{\sigma}_j^{(i)})^2\right)\right)^{1/2},
\end{eqnarray*}
where $R$ is a tubal rank, $P\geq 2$ is an oversampling parameter, $q$ is the power iteration, $\widetilde{\sigma}_j^{(i)}$
is the $i$-th component of ${\rm fft}(\underline{\bf S}(j,j,:),[ ],3)$, and the singular value gap $\widetilde{\tau}^{(i)}_R=\frac{\widetilde{\sigma}_{R+1}^{(i)}}{\widetilde{\sigma}_R^{(i)}}\ll 1$.
\end{thm}

\begin{thm}\label{thm_4}
 Given an $I_1\times I_2\times I_3$ tensor $\underline{\bf X}$ and an $I_2\times (R+P) \times I_3$ Gaussian tensor $\underline{\bf \Omega}$, if $\underline{\bf Q}$ is obtained
from Algorithm \ref{ALg_PASEFF}, then
\begin{eqnarray*}
\mathbb{E}\left(\|{\cX}-{\cX}*{\cQ}*{\cQ}^T\|_F\right)
\leq\hspace*{3.5cm}\\
\left(\frac{1}{I_3}\sum_{i=1}^{I_3}\left(1+\frac{R}{P-1}(\widetilde{\tau}^{(i)}_R)^{2(2v-1)}\right)
\left(\sum_{j>R}(\widetilde{\sigma}_j^{(i)})^2\right)\right)^{1/2},
\end{eqnarray*}
where $R$ is the tubal rank, $P\geq 2$ is the oversampling parameter, with an odd number of passes $v$, $\widetilde{\sigma}_j^{(i)}$
is the $i$-th component of ${\rm fft}(\underline{\bf S}(j,j,:),[ ],3)$, and the singular value gap $\widetilde{\tau}^{(i)}_R=\frac{\widetilde{\sigma}_{R+1}^{(i)}}{\widetilde{\sigma}_R^{(i)}}\ll 1$.
\end{thm}

\begin{proof}
See the Appendix.
\end{proof}

\section{An Application to tensor completion}\label{Sec:APP}
{
The problem of recovering a data tensor from only a part of its components is known as {\it tensor completion} \cite{song2019tensor}. Let $\cX\in\mathbb{R}^{I_1\times I_2\times \cdots\times I_N}$ be a given tensor with missing elements, where the indicator set $\underline{\bf\Omega},$ stores the location of known (observed) elements.
It is generally known that we may effectively recover the underlying original tensor from its incomplete version under the low-rank property assumption \cite{song2019tensor}. 
The following is a common tensor decomposition approach to solve the tensor completion problem
\cite{candes2009exact,song2019tensor}
\begin{equation}\label{MinRankCompl2}
\begin{array}{cc}
\displaystyle \min_{\underline{\bf X}} & {\|{{\bf P}_{\underline{\bf\Omega}} }({\underline{\bf X}})-{{\bf P}_{\underline{\bf\Omega}} }({\underline{\bf M}})\|^2_F},\\
\textrm{s.t.} & {\rm Rank}(\underline{\bf X})=R,\\
\end{array}
\end{equation}
where $\underline{\bf M}$ is the exact data tensor. As described in \cite{ahmadi2023fast}, using an auxiliary variable ${\underline{\bf C}}$, the optimization problem \eqref{MinRankCompl2} can be solved more conveniently by the following reformulation
\begin{equation}\label{MinRankCompl3}
\begin{array}{cc}
\displaystyle \min_{\underline{\bf X},\underline{\bf C}} & {\|{\underline{\bf X}}-{\underline{\bf C}}\|^2_F},\\
\textrm{s.t.} & {\rm Rank}(\underline{\bf X})=R,\\
& {{\bf P}_{\underline{\bf\Omega}} }({\underline{\bf C}})={{\bf P}_{\underline{\bf\Omega}} }({\underline{\bf M}})\\
\end{array}
\end{equation}
and we can alternatively solve the minimization problem \eqref{MinRankCompl2} over variables $\underline{\bf X}$ and $\underline{\bf C}$. Thus, the solution to the minimization problem \eqref{MinRankCompl2} can be approximated by the following iterative procedures
\begin{equation}\label{Step1}
\underline{\mathbf X}^{(n)}\leftarrow \mathcal{L}(\underline{\mathbf C}^{(n)}),
\end{equation}
\begin{equation}\label{Step2}
\underline{\mathbf C}^{(n+1)}\leftarrow\underline{\mathbf \Omega}\oast\underline{\mathbf M}+(\underline{\mathbf 1}-\underline{\mathbf \Omega})\oast\underline{\mathbf X}^{(n)},
\end{equation}
where $\mathcal{L}$ is an operator to compute a low-rank tensor approximation of the data tensor $\underline{\mathbf C}^{(n)}$ and $\underline{\mathbf 1}$ is a tensor whose all components are equal to one. Note that equation \eqref{Step1} solves the minimization problem \eqref{MinRankCompl3} over  
$\underline{\bf X}$ for fixed variable $\underline{\bf C}$. Also, Equation \eqref{Step2} solves the minimization problem \eqref{MinRankCompl3} over  
$\underline{\bf C}$ for fixed variable $\underline{\bf X}$. The algorithm consists of two main steps, {\it low-rank tensor approximation} \eqref{Step1} and {\it Masking computation} \eqref{Step2}. It starts from the initial incomplete data tensor $\underline{\mathbf X}^{(0)}$ with the corresponding observation index set $\underline{\mathbf \Omega}$ and  sequentially improves the approximate solution till some stopping criterion is satisfied or the maximum number of iterations is reached. We do not need to compute the term $\underline{\mathbf \Omega}\oast\underline{\mathbf M}$ at each iteration because it is just the initial data tensor $\underline{\mathbf X}^{(0)}$. The filtering/smoothing procedure is a known technique in signal processing community to improve the image quality. Indeed, in the above procedure, we exploit this idea and smooths out the tensor $\underline{\bf C}^{(n+1)}$ before applying the low tensor rank approximation operator $\mathcal{L}$ to get better results. The first step is computationally expensive steps especially when a large number of iterations is required for convergence or the data tensor is quite large. Here, we use our randomized pass-efficient Algorithm \ref{ALg_PASEFF} instead of the deterministic algorithms. The experimental results show that this algorithm provides promising results with lower computational cost.}

\section{Simulations}\label{Sec:Sim}
In this section, we test the proposed randomized algorithm on synthetic and real-world datasets. All numerical simulations were performed on a laptop computer with 2.60 GHz Intel(R) Core(TM) i7-5600U processor and 8GB memory. The Peak Signal-to-Noise Ratio (PSNR) and relative error have been utilized to evaluate the performance of the proposed algorithm. The PSNR of two images $\underline{\bf X}$ and $\underline{\bf Y}$ is defined as
\[
{\rm PSNR}=10\log 10\left({\frac{{\|\cX\|}_{\infty}}{\|\cX-\cY\|_F}}\right)\quad {\rm dB}.
\]
The relative error is also defined as 
\[
e(\widetilde{\cX})=\frac{\|\cX-\widetilde{\cX}\|_F}{\|\cX\|_F},
\]
where $\X$ is the original tensor and $\widetilde{\X}$ is the approximated tensor. {We compare the proposed algorithm with the baseline methods: Truncated t-SVD (Algorithm \ref{ALG:t-SVD}) and randomized t-SVD (Algorithm \ref{ALgRRM}).}
\begin{exa}\label{Ex_1}
In this experiment, we generate a tensor $\cX\in\mathbb{R}^{500\times 500\times 500}$ with exact tubal rank $15$. We set the oversampling parameter $P=5$ and apply Algorithm \ref{ALg_PASEFF} with the tubal rank $R=10$ with different numbers of passes over the data tensor $\cX$. In Figure \ref{Pic_Exam_1} (right), we report the relative error versus the number of passes. The results show that with $v=2,3$ passes, we can achieve quite good results and for a larger number of passes, the relative error is smoothly decreased, while the computational complexity is also higher. In Figure \ref{Pic_Exam_1} (left), we report the running time of Algorithm \ref{ALg_PASEFF} for different numbers of passes.  Please note that the benefit of Algorithm \ref{ALg_PASEFF} compared with Algorithm \ref{ALgRRM} is that it does not necessarily need to pass the original data tensor $\cX$ four times, and with only two passes, we can achieve reasonable results. This is one benefit of the proposed method over method \ref{ALgRRM}, which lacks this flexibility in terms of the number of passes. In Table \ref{Table1}, we compare the outcomes obtained by the proposed method with those obtained by the truncated t-SVD (Algorithm \ref{ALG:t-SVD}) and the randomized t-SVD (Algorithm \ref{ALgRRM}). The outcome clearly demonstrates that we can get about the same accuracy for the aforementioned data tensor in much less time. It is intriguing that the suggested approach only requires two runs, but Algorithm \ref{ALgRRM} requires at least four passes (for power iteration q = 1). This means that the proposed algorithm even requires less running time than the randomized Algorithm \ref{ALgRRM}. 
\begin{table}
\begin{center}
\caption{Comparing the running time and relative errors achieved by the proposed algorithm, Algorithm \ref{ALG:t-SVD} and Algorithm \ref{ALgRRM} for Example \ref{Ex_1}. The results are for the tubal rank $R=10$.}\label{Table1}
\vspace{0.2cm}
{\smaller
\begin{tabular}{||c| c | c ||} 
 \hline
Algorithms  & Running Time (Seconds) & Relative error \\
 \hline\hline
 Truncated t-SVD \cite{kilmer2011factorization,kilmer2013third} &  25.52 &  {\bf 3.4e-15}  \\ 
  Randomized t-SVD \cite{zhang2018randomized} &  17.65 &  5.2e-15  \\ 
 Proposed algorithm  & {\bf 8.23}  &  7.1e-15  \\
 \hline
\end{tabular}
}
\end{center}
\end{table}

The output of Algorithm \ref{ALg_PASEFF} for two and three runs was insignificant since the artificial data tensor utilized in this case was noiseless. Examples \ref{salman} and \ref{salman_2} show that the outcomes of two and three passes are important for the real-world datasets (images and videos). 

To further examine the proposed approach, the results for different tubal ranks are also reported in Figure \ref{ex1_dr}. We see that the proposed algorithm is still efficient and robust for other tubal ranks. 

{Additionally, we examined the robustness of the proposed algorithm for situations that the true rank is not selected. To this end, we first considered a larger tubal rank $R=15$, where the algorithm provided an approximation with a relative error of $2.1231e-15$ while for a smaller tubal rank $R=5$, the algorithm gave an approximation with a relative error of $0.5619$ which is close to the best approximation computed by the t-SVD. To solve the problem of selecting a tubal rank lower than the true tubal rank, one can gradually increase the tubal rank until a satisfying approximation is achieved. Indeed, we combined this technique with our proposed algorithm and by gradually increasing the tubal rank from $R=5$ to $R=10$, the approximation with a relative error of $5.4582e-15$ was achieved. As the proposed algorithm is very fast especially for relatively small tubal ranks, it is also applicable for the exact tubal rank estimation task. These simulations convinced us that it can be efficiently used in different applications.}

\end{exa}

\begin{figure*}
\begin{center}
\includegraphics[width=0.5\columnwidth]{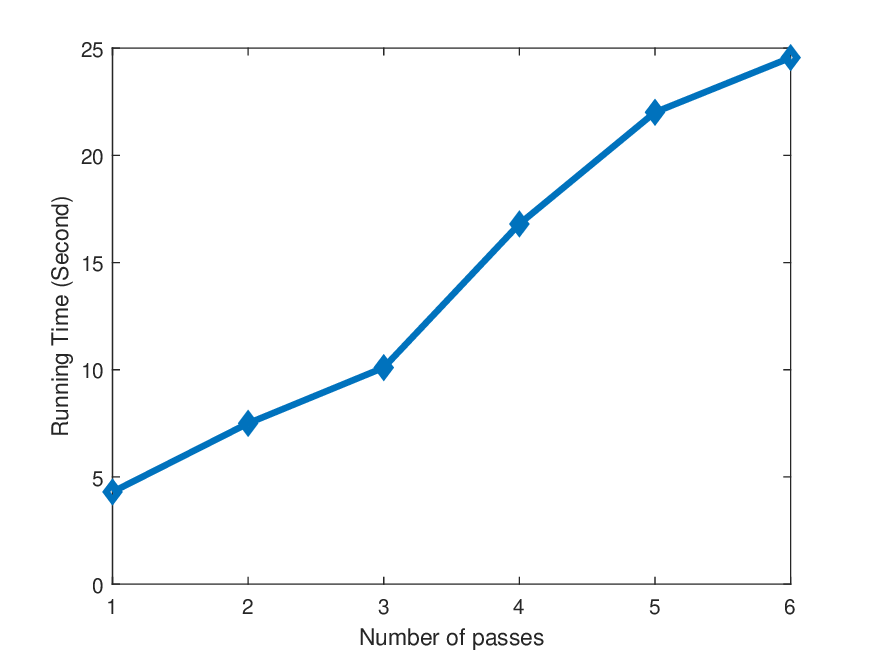}\includegraphics[width=0.5\columnwidth]{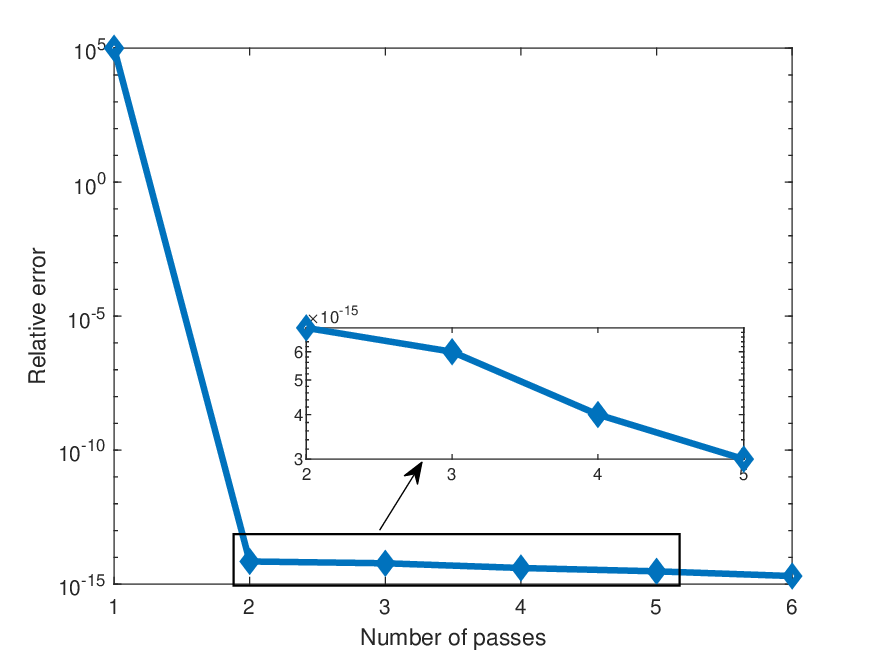}\\
\caption{\small{Running time and relative error of the approximations achieved by the proposed algorithm for a synthetic data tensor of size $500\times 500\times 500$ and the tubal rank $R=15$ using different numbers of passes for Example \ref{Ex_1}.}}\label{Pic_Exam_1}
\end{center}
\end{figure*}

\begin{figure*}
\begin{center}
\includegraphics[width=0.53\columnwidth]{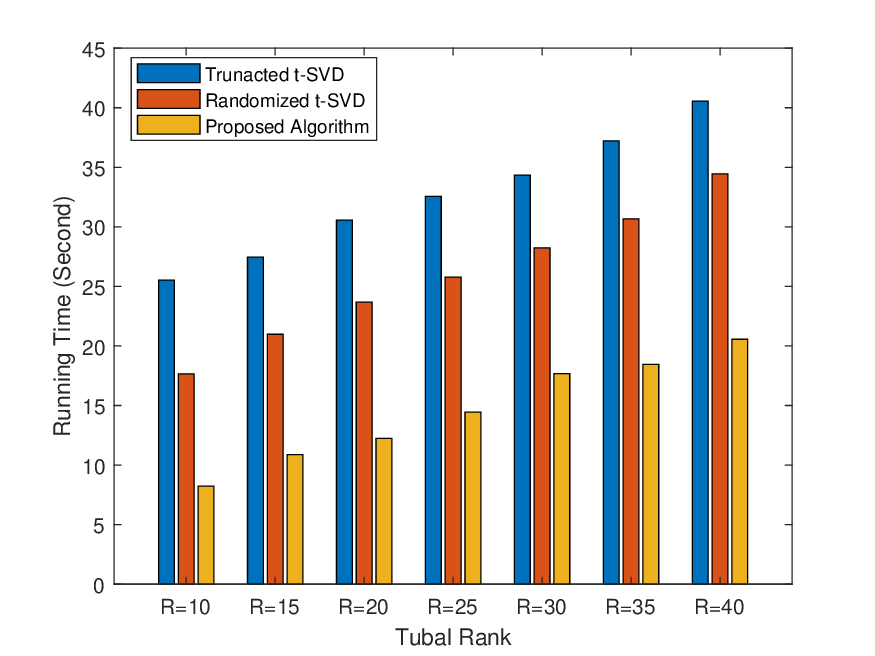}\includegraphics[width=0.53\columnwidth]{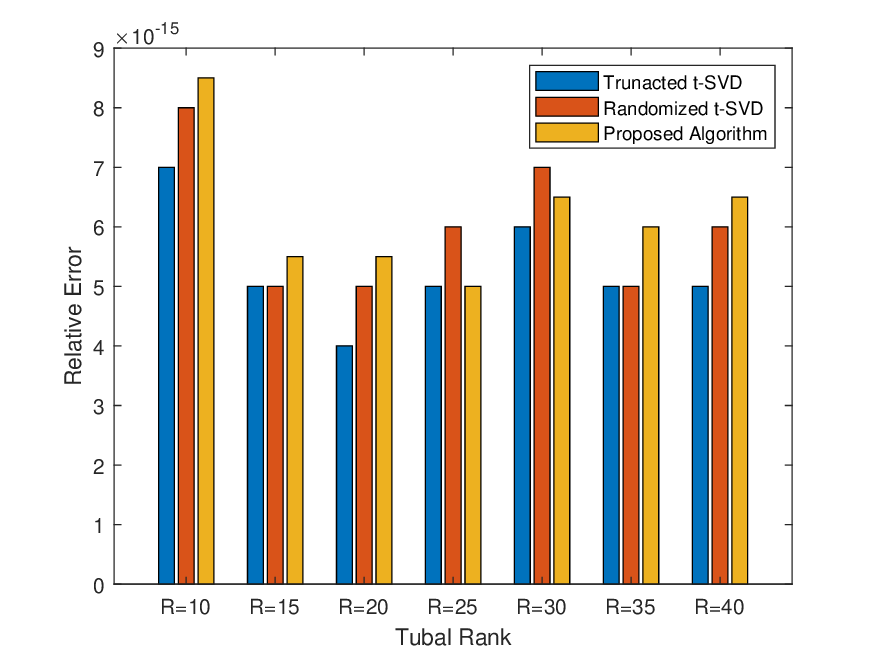}\\
\caption{\small{Running time and relative error of the approximations achieved by the truncated t-SVD, the randomized t-SVD and the proposed algorithm for a synthetic data tensor of size $500\times 500\times 500$ for different tubal ranks for Example \ref{Ex_1}.}}\label{ex1_dr}
\end{center}
\end{figure*}

\begin{exa}\label{salman}
In this example, we apply the proposed algorithm to compress color images. To this end, we consider the ``Kodim03'' and ``Kodim23'' color images included in the Kodak dataset \footnote{\url{http://www.cs.albany.edu/~xypan/research/snr/Kodak.html}}. We set the oversampling parameter $P=6$ and apply Algorithm \ref{ALg_PASEFF} to the mentioned images with the tubal rank $R=40$ for different numbers of passes. The reconstructed images and corresponding results including relative error, PSNR, and running time achieved by the proposed algorithm are reported in Figure \ref{Pic_Exam_2_1}. As can be seen, the higher the number of passes, the better performance of the images and the higher the computational cost. Moreover, we compare the running time and the PSNR achieved by the proposed algorithm with the Truncated t-SVD (Algorithm \ref{ALG:t-SVD}) and the randomized t-SVD (Algorithm \ref{ALgRRM}) and they are shown in Figure \ref{Table_Salman}. These results indicate that the proposed algorithm can provide approximately the same reconstruction (PSNR) as the baseline methods but with less execution time. 
\begin{table}
\begin{center}
\caption{Comparing the running time and the PSNR achieved by the proposed algorithm, Algorithm \ref{ALG:t-SVD} and Algorithm \ref{ALgRRM} for Example \ref{salman}. The results are for the tubal rank $R=40$.}\label{Table_Salman}
\vspace{0.2cm}
{\smaller
\begin{tabular}{||c| c | c ||} 
 \multicolumn{3}{c}{Kodim23 Image}\\
 \hline
Algorithms  & Running Time (Seconds) & PSNR \\
 \hline\hline
 Truncated t-SVD \cite{kilmer2011factorization,kilmer2013third} & 0.45  &  {\bf 27.87}  \\ 
  Randomized t-SVD \cite{zhang2018randomized} & 0.23  & 27.51   \\ 
 Proposed algorithm  & {\bf 0.18}  &   27.38 \\
 \hline
 \multicolumn{3}{c}{Kodim03 Image}\\\hline
 Algorithms  & Running Time (Seconds) & PSNR \\
 \hline\hline
 Truncated t-SVD \cite{kilmer2011factorization,kilmer2013third} &  0.23 &  {\bf 27.67}  \\ 
  Randomized t-SVD \cite{zhang2018randomized} &  0.19 &  27.39  \\ 
 Proposed algorithm  & {\bf 0.10}  &  27.23  \\
 \hline
\end{tabular}
}
\end{center}
\end{table}

\begin{figure*}
\begin{center}
\includegraphics[width=.9\columnwidth]{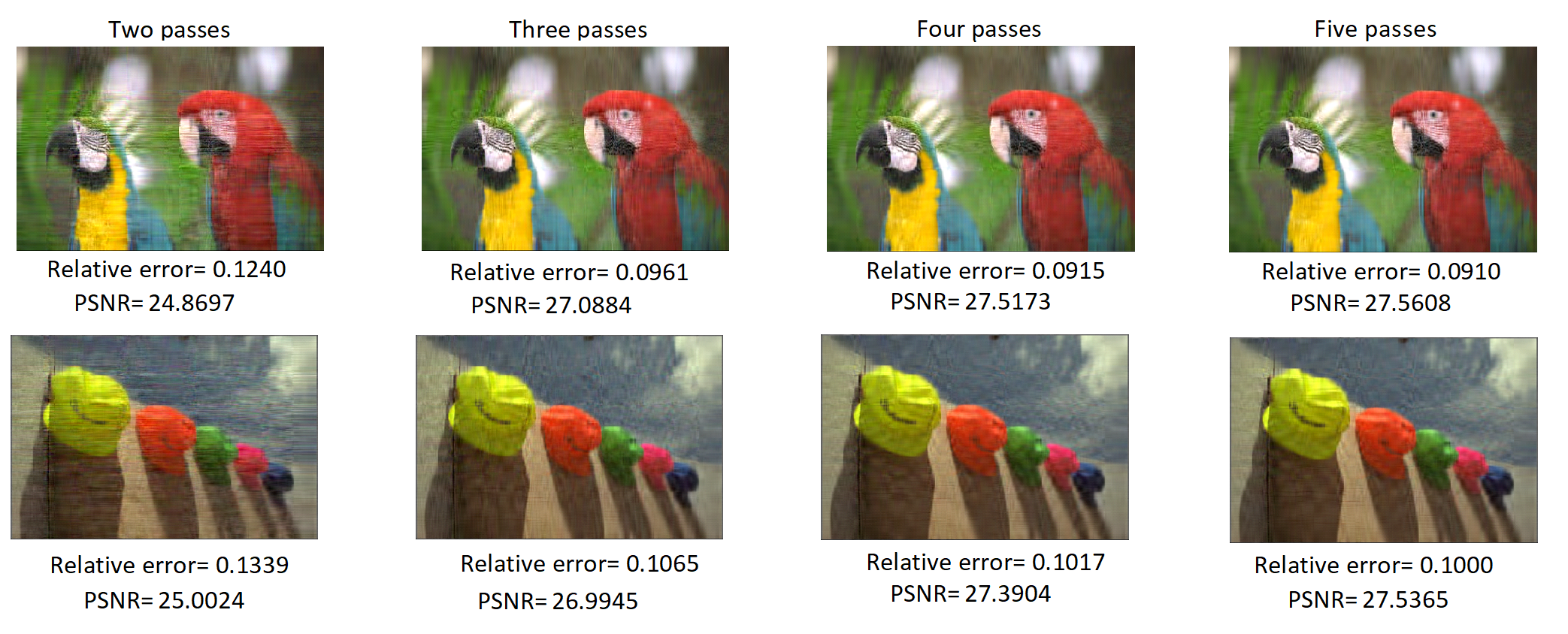}\\
\includegraphics[width=.5\columnwidth]{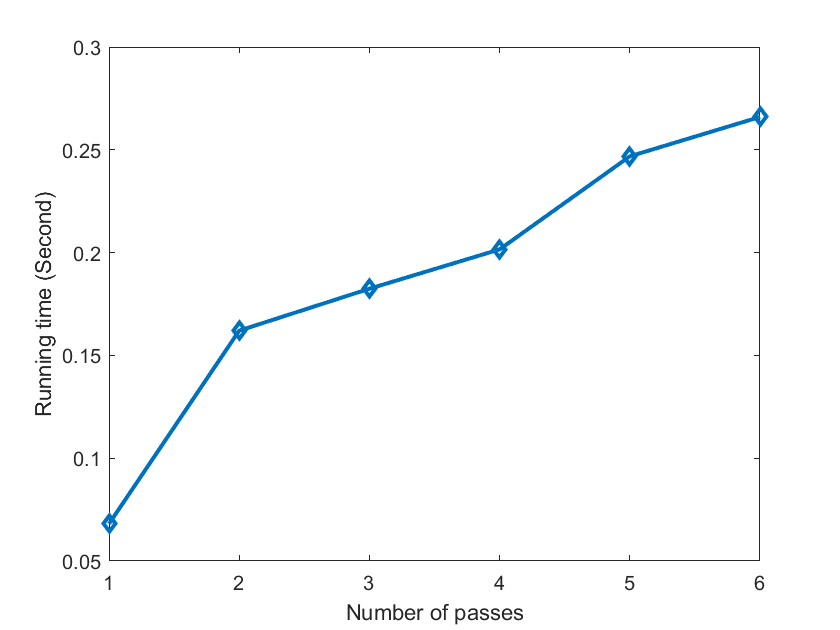}\includegraphics[width=.5\columnwidth]{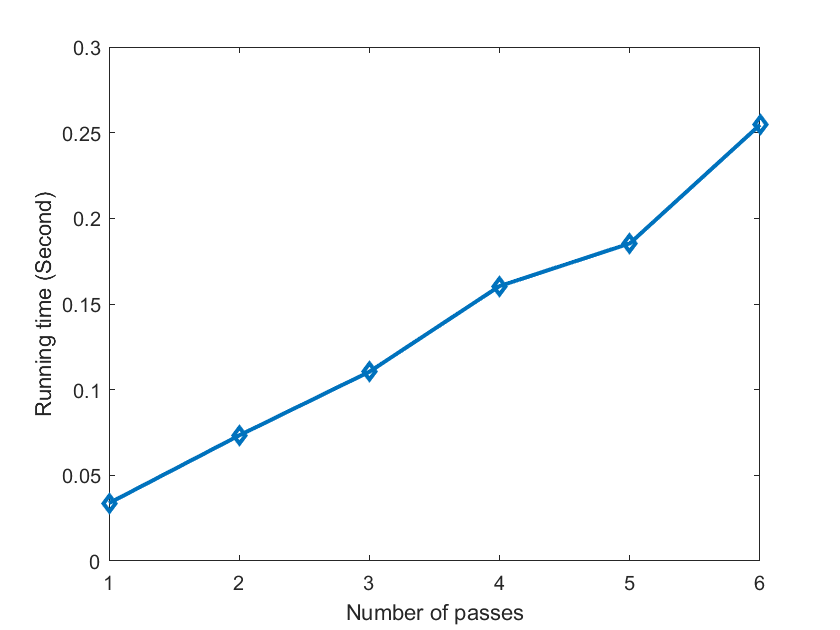}
\caption{\small{(Upper) The reconstruction of the ``Kodim23'' and the ``Kodim03'' images using the proposed algorithm for the tubal rank $R=20$ and different numbers of passes\,\,(Bottom) The running time of the proposed algorithm for computation of the truncated t-SVD of the ``Kodim23'' image (left) and the ``Kodim03'' image (right) with the tubal rank $R=40$ and using different numbers of passes for Example \ref{salman}. .}}\label{Pic_Exam_2_1}
\end{center}
\end{figure*}

\begin{figure*}
\begin{center}
\includegraphics[width=0.53\columnwidth]{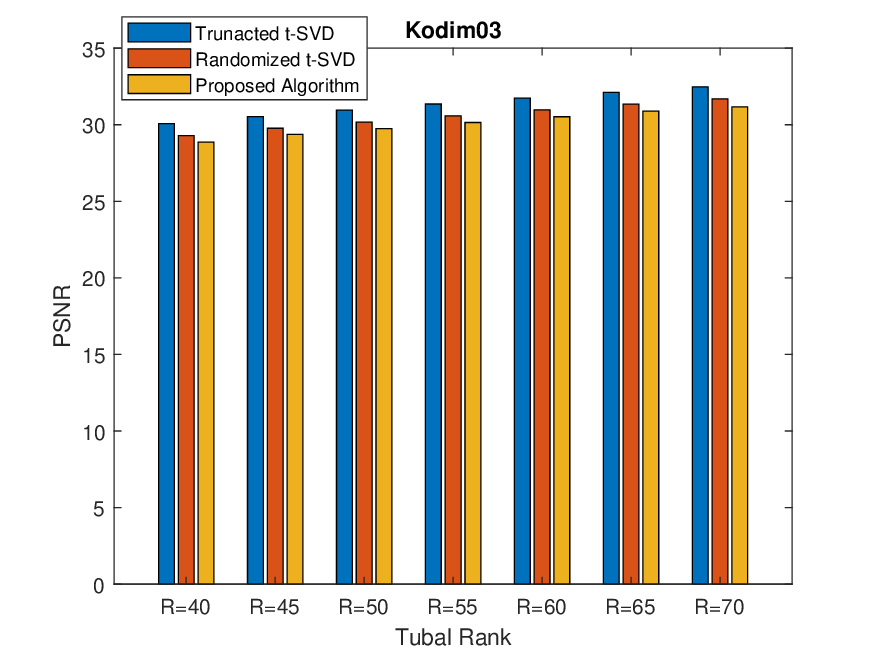}\includegraphics[width=0.53\columnwidth]{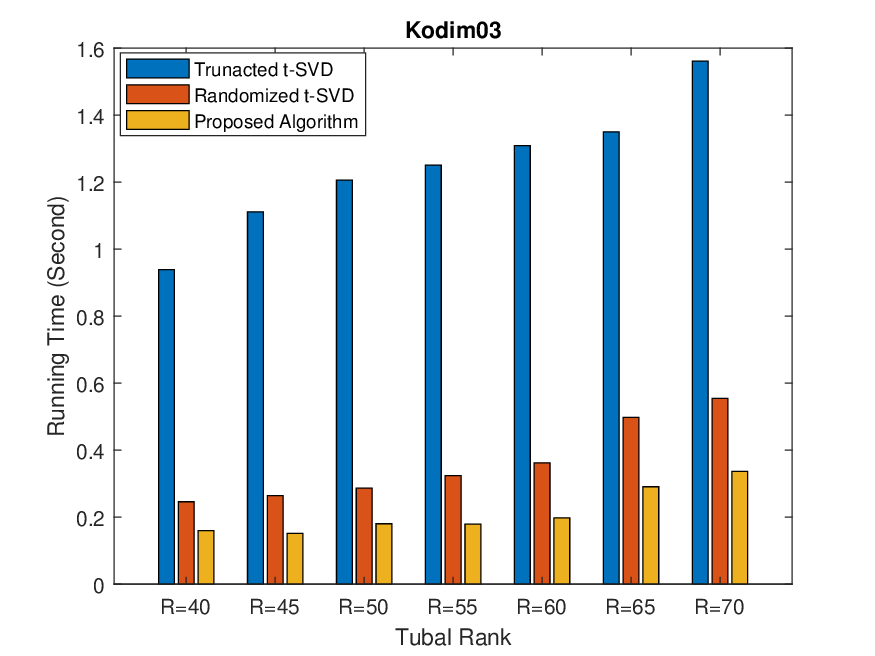}\\
\includegraphics[width=0.53\columnwidth]{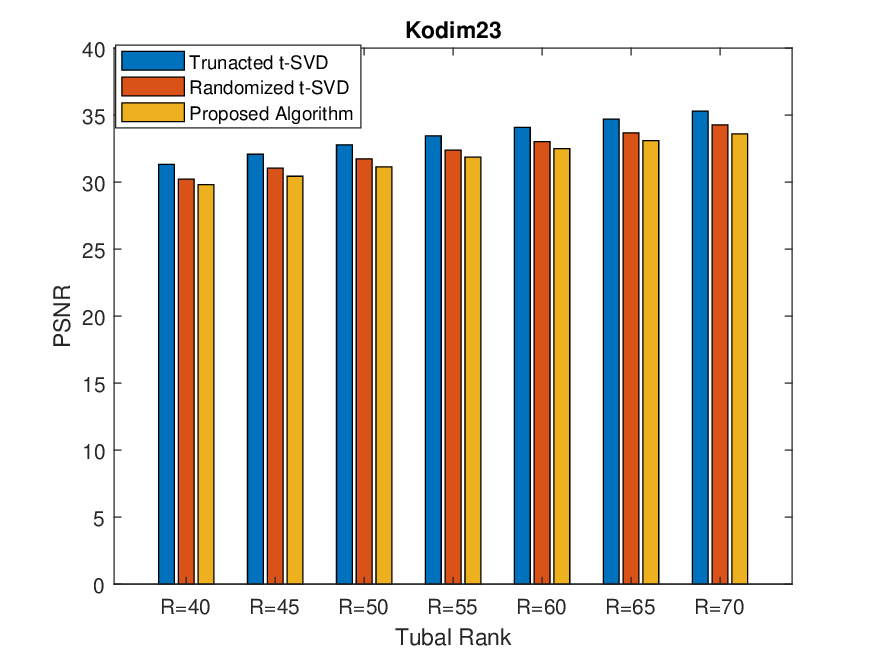}\includegraphics[width=0.53\columnwidth]{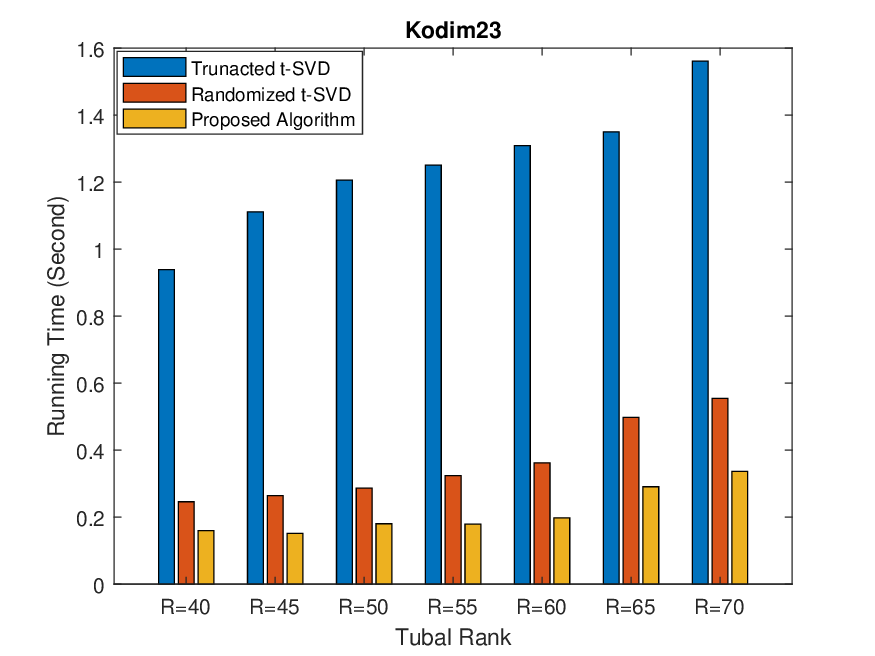}\\
\caption{\small{The running time and PSNRs of the reconstructed images achieved by the truncated t-SVD, the randomized t-SVD and the  proposed algorithm using different tubal ranks for Example \ref{salman}.}}\label{ex2_drm}
\end{center}
\end{figure*}

\end{exa}

\begin{exa}\label{salman_2}
In this experiment, we examine Algorithm \ref{ALg_PASEFF} for compressing video datasets. We have used the ``Foreman'' and ``News'' videos from \cite{WinNT} in this test. Both videos are third-order tensors of size $144\times 176\times 300$. We set the oversampling parameter $P=5$ and apply the proposed algorithm for computing t-SVD with the tubal rank $R=20$. For this tubal rank, we achieve the compression ratio $3.7271$. The PSNR of some random samples of the frames for the mentioned two videos and different numbers of passes are reported in Figure \ref{Video_recon_1}. Besides, the PSNR of all frames of both videos are shown in Figure \ref{Video} (upper). The corresponding running times can been seen in Figure \ref{Video} (bottom). Here again, the same results as the previous two simulations are achieved and using more passes over the video dataset, we achieve better results with a higher computational cost. A comparison between the mean of the PSNR of all reconstructed frames by the proposed algorithm and the truncated t-SVD (Algorithm \ref{ALG:t-SVD}) and the randomized t-SVD (Algorithm \ref{ALgRRM}) are made in Table \ref{Table_Salman_2}. The running time and PSNRs of the reconstructed images obtained by the proposed algorithm and the baselines for different tubal ranks are also compared in \ref{ex2_drm}. Here again, we see that the proposed algorithm provides almost the same reconstruction as the baseline but with less running time.
\end{exa}

\begin{figure*}
\begin{center}
\includegraphics[width=.9\columnwidth]{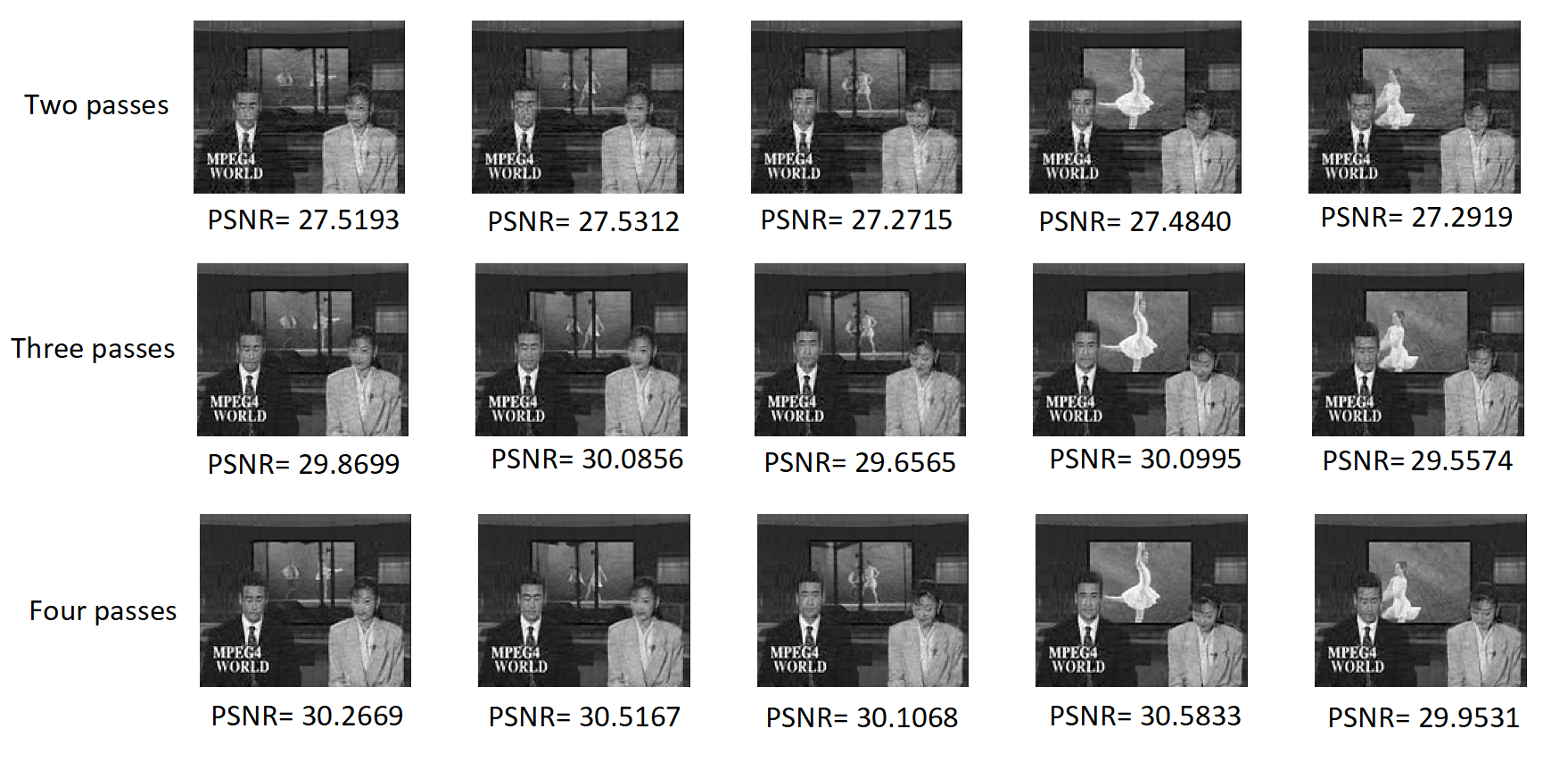}\\\includegraphics[width=.9\columnwidth]{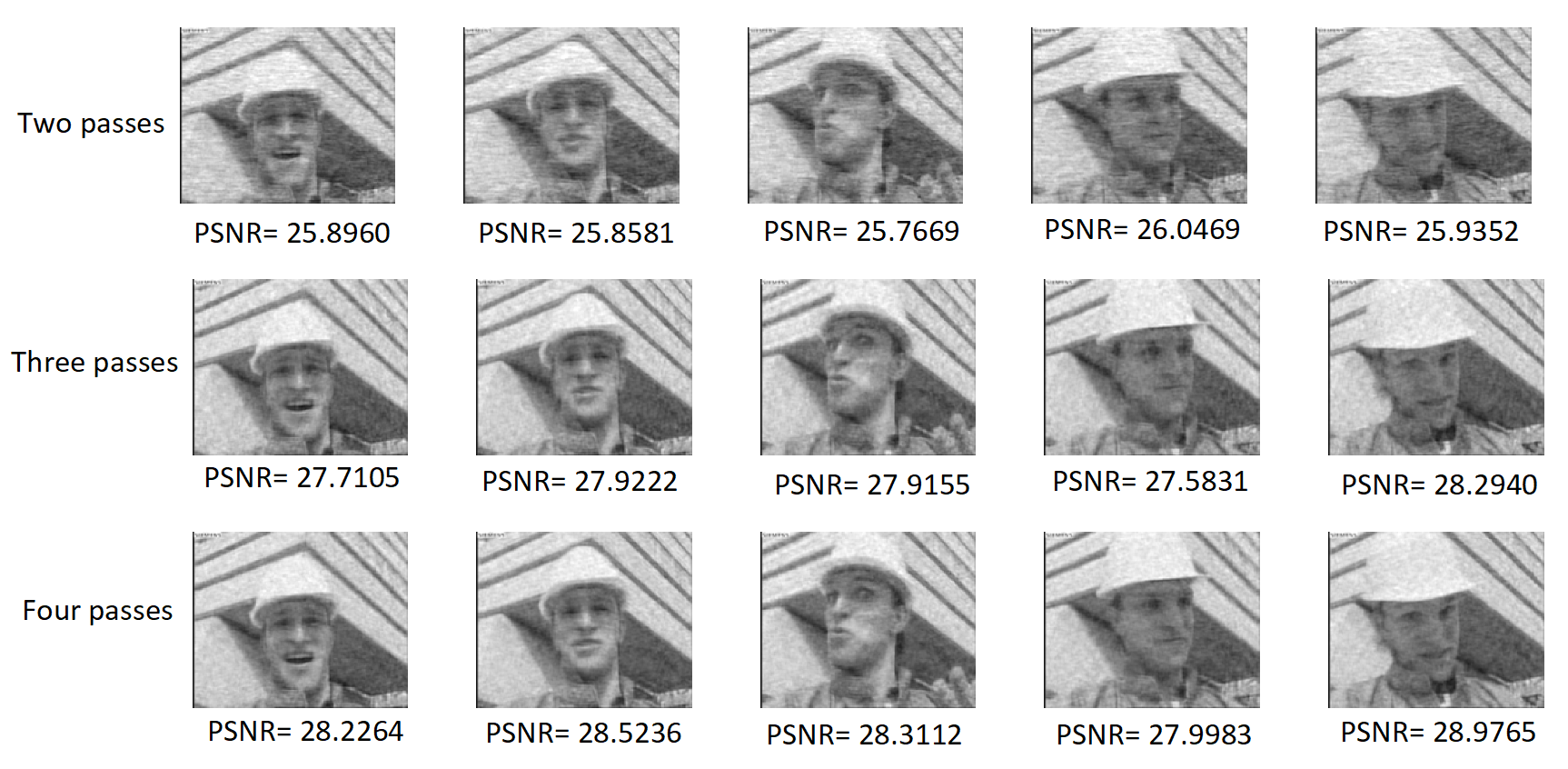}
\caption{\small{Reconstruction of some random frames for the ``News'' and the ``Foreman'' videos using the proposed algorithm. The tubal rank $R=20$ and different numbers of passes for Example \ref{salman_2}.}}\label{Video_recon_1}
\end{center}
\end{figure*}

\begin{table}
\begin{center}
\caption{Comparing the running time and the mean of the PSNR of all frames achieved by the proposed algorithm, Algorithm \ref{ALG:t-SVD} and Algorithm \ref{ALgRRM} for Example \ref{salman_2}. The results are for the tubal rank $R=25$. }\label{Table_Salman_2}
\vspace{0.2cm}
 {\smaller
\begin{tabular}{||c| c | c ||} 
 \multicolumn{3}{c}{News dataset}\\
 \hline
Algorithms  & Running Time (Seconds) & PSNR \\
 \hline\hline
 Truncated t-SVD \cite{kilmer2011factorization,kilmer2013third} & 2.64  &  {\bf 30.45}  \\ 
  Randomized t-SVD \cite{zhang2018randomized} & 2.02  & 30.10   \\ 
 Proposed algorithm  & {\bf 1.10}  &   29.65 \\
 \hline
 \multicolumn{3}{c}{Foreman dataset}\\
 \hline
 Algorithms  & Running Time (Seconds) & PSNR \\ \hline\hline
 Truncated t-SVD \cite{kilmer2011factorization,kilmer2013third} &  2.67 &  {\bf 28.25}  \\ 
  Randomized t-SVD \cite{zhang2018randomized} &  2.13 &  28.31  \\ 
 Proposed algorithm  & {\bf 1.05}  &  28.09  \\
 \hline
\end{tabular}
}
\end{center}
\end{table}

\begin{figure*}
\begin{center}
\includegraphics[width=.5\columnwidth]{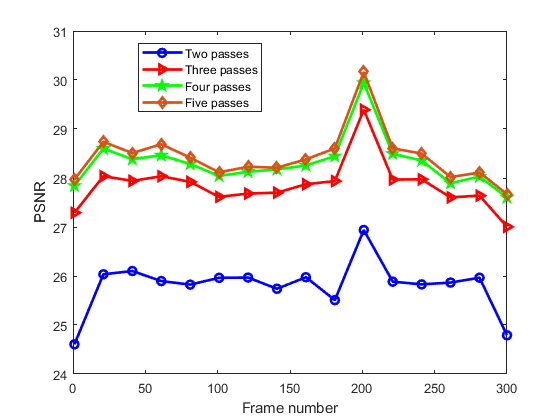}\includegraphics[width=.5\columnwidth]{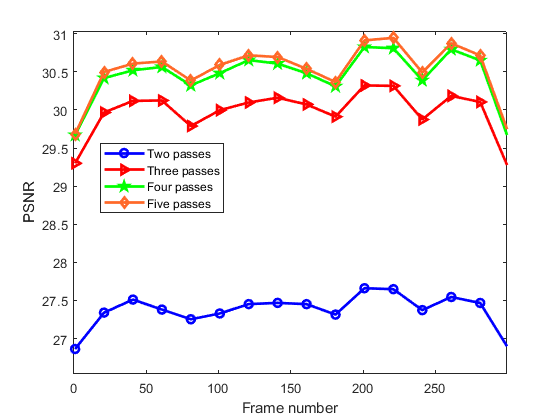}\\
\includegraphics[width=.5\columnwidth]{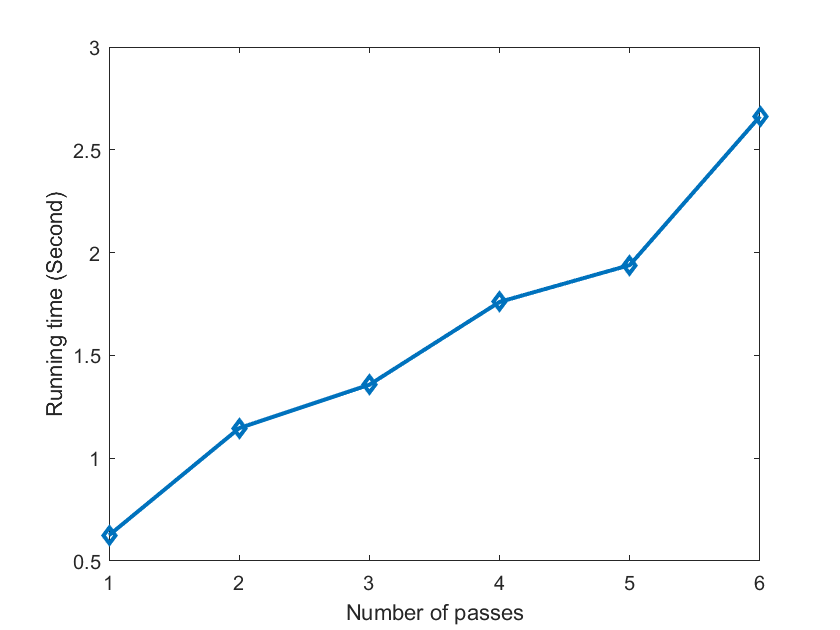}\includegraphics[width=.5\columnwidth]{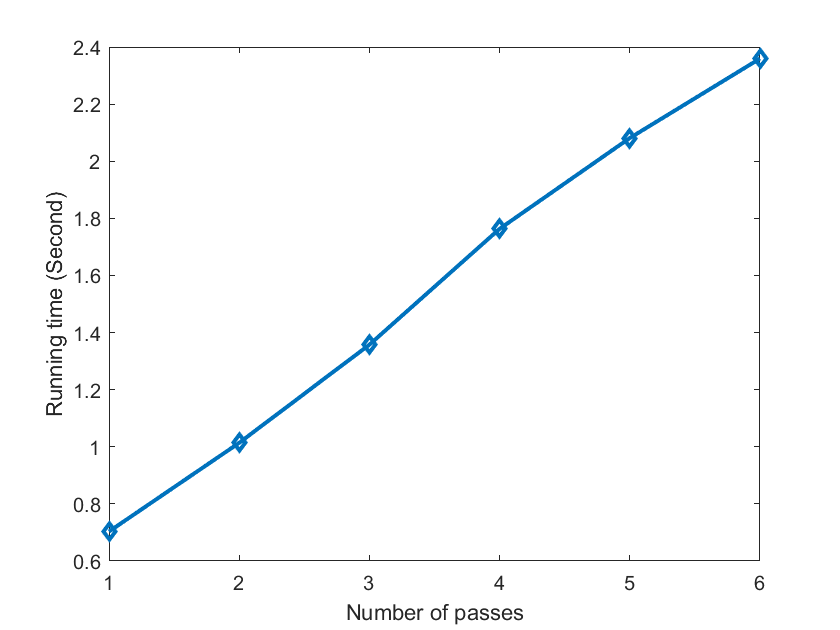}
\caption{Example \ref{salman_2}. \small{(Upper) The results of the proposed algorithm for video compression using different numbers of passes and the tubal rank $R=20$ (the left figures are for the ``Foreman'' video and the right is for the ``News'' video). (Bottom). The running time of the proposed algorithm for different numbers of passes and the tubal rank $R=20$. The left figure is for the ``Foreman'' video and the right figure is for the ``News'' video}.}\label{Video}
\end{center}
\end{figure*}

\begin{figure*}
\begin{center}
\includegraphics[width=0.53\columnwidth]{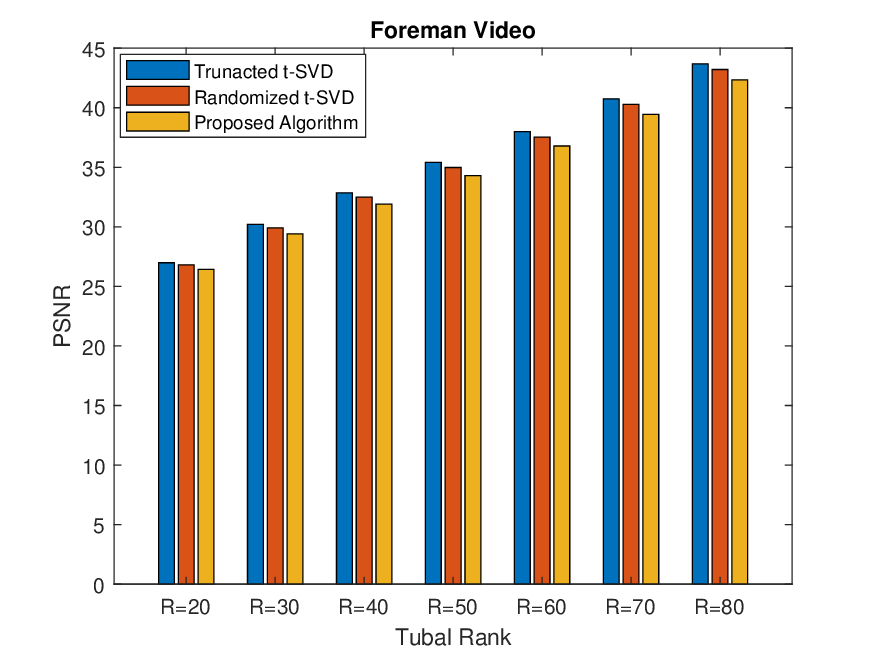}\includegraphics[width=0.53\columnwidth]{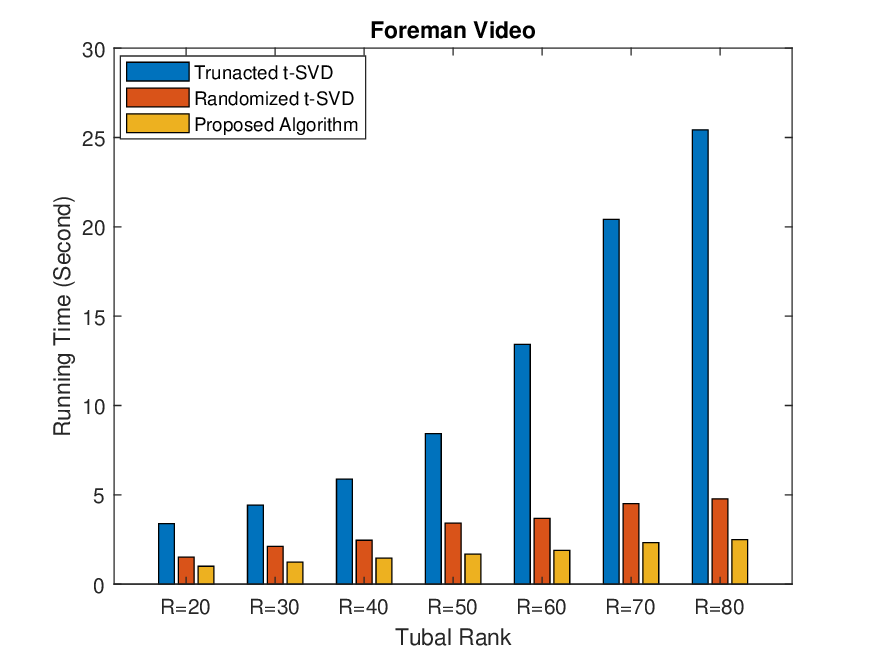}\\
\includegraphics[width=0.53\columnwidth]{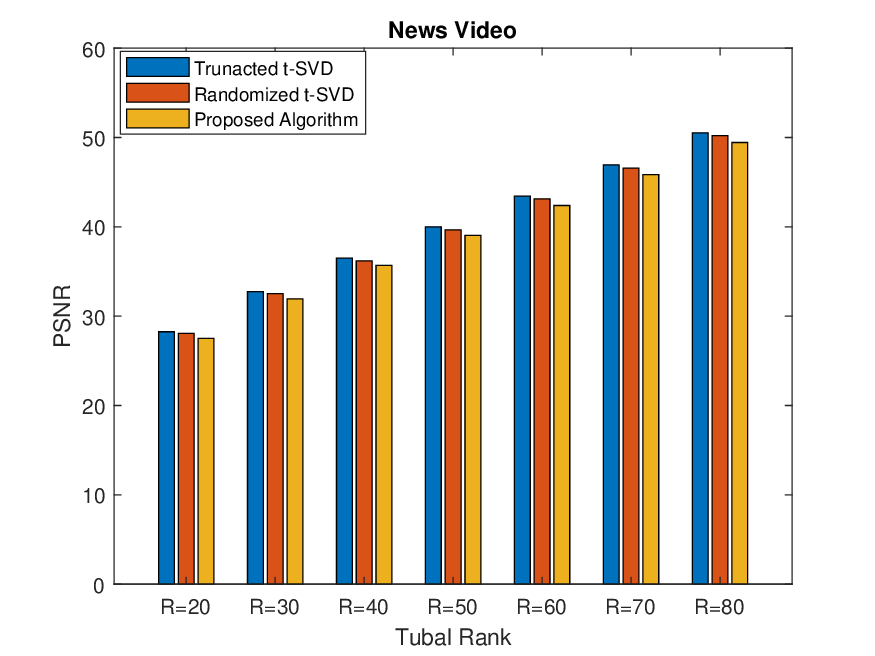}\includegraphics[width=0.53\columnwidth]{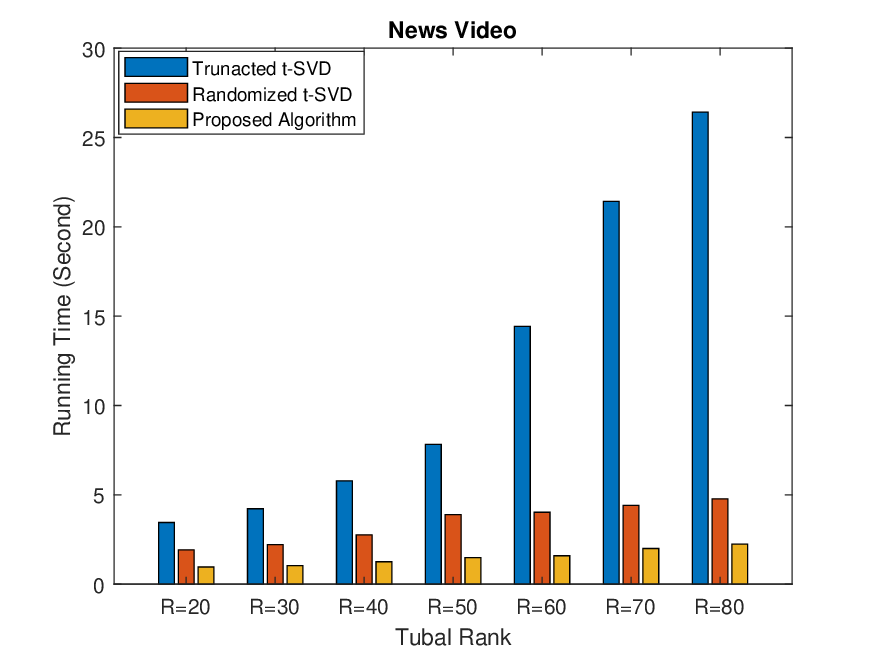}\\
\caption{\small{The running time and the mean of the PSNRs of the all reconstructed frames of the videos achieved by the truncated t-SVD, the randomized t-SVD and the proposed algorithm using different tubal
ranks \ref{salman_2}.}}\label{ex3_drm}
\end{center}
\end{figure*}

\begin{exa}\label{salman_com}
In this simulation we investigate he effectiveness of Algorithm \ref{ALg_PASEFF} for the tensor completion task described in Section \ref{Sec:APP}.
 In our simulations, we make use of both images and videos. First, we consider four color images taken from the well-known Kodak dataset (``Kodim03'', ``Kodim15'', ``Kodim16'', ``Kodim23''). The size of these images is $512\times 768\times 3,$ and we remove $80\%$ of the pixels randomly. Then apply the completion procedure described in Section \ref{Sec:APP} with the tubal rank $R=30$ and use Algorithm \ref{ALg_PASEFF} with two passes ($v=2$) and the oversampling parameter $P=10$. The reconstructed images and their PSNRs are displayed in Figure \ref{Im_1}. The results show the good performance of the proposed algorithm for the image completion task. A comparison between the proposed algorithm and Truncated t-SVD (Algorithm \ref{ALG:t-SVD}) and randomized t-SVD (Algorithm \ref{ALgRRM}) for the low tubal rank approximation is made and the results are shown in Table \ref{Table_Salman_imacom}. The experimental results clearly illustrate that the proposed algorithm can provide the recovered images with approximately the same accuracy but much faster.
 
 To compare the pass-efficient tensor-based (Algorithm \ref{ALg_PASEFF}) and the pass-efficient matrix-based algorithms (Algorithm \ref{ALg_2}), we consider the pepper image depicted in Figure \ref{Im_com_t_m} (first left) that is of size $256\times 256\times 3$ and we remove some pixels of the image in a structure way shown in the second left image. Then, we apply the completion procedure described in Section where the proposed pass-efficient tensor-based algorithm (Algorithm \ref{ALg_PASEFF}) was used for the operator $\mathcal{L}$. Also, we reshaped the image to a matrix and then apply the iterative procedures \eqref{Step1}-\ref{Step2} to it, in which the pass-efficient matrix-based algorithm (Algorithm \ref{ALg_2}) is used or the operator $\mathcal{L}$. These reconstructed images using two algorithms are visualized in Figure \ref{Im_com_t_m}. This indicates that due to the tensor
structure-preserving property, the proposed algorithm has overall better reconstruction quality. {The proposed algorithm for a given tubal rank $R$, provides close to optimal approximation, which is achieved by the truncated t-SVD. However, in some image restoration methods, like matrix factorization-based techniques, the rank ($R$) of a matrix or a tensor is a crucial parameter. This parameter represents the number of significant components or features used to model the image. Choosing an appropriate rank is essential, and it can significantly impact the quality of the restored image.  If the user provides an incorrect value for R (either too large or too small), it can negatively affect the quality of the restored image. A too small value might result in a loss of details, while a too large value may overfit the data, leading to artifacts and noise into the image.} 

{We used the idea known in signal processing, where one starts with a small tubal rank and gradually increases it until the quality of the image is not improved significantly\cite{zhao2015bayesian}. We should highlight that depending on the specific characteristics of the image and the degradation process, the initial tubal ranks may be different. For example, in our experiments reported in the paper, for images of size $512\times 768\times 3$ and $80\%$ of pixels removed, we started by the tubal rank of $25$ and gradually increased the tubal rank for which the tubal rank of $30$ provided the best reconstruction. Also, for the images with 90\% of pixels missing, we started with the tubal rank of $R=15$ and again gradually increased the tubal rank. Here the tubal rank $R=21$ provided the best recovery results. Please note that since the proposed randomized algorithm is fast, running it for several tubal ranks is not a concern and this is one of the advantages of the proposed algorithm. } 
 
 We next considered three videos ``Akiyo'', ``Foreman'', and ``News''\cite{WinNT}, which are third-order tensors of size  $176\times 144\times 300$ (collocation of $300$ frames or black and white images). We remove $70\%$ of the pixels of the mentioned videos randomly. With the same procedure described for the images, we used our proposed pass-efficient algorithm in the completion  stage \eqref{Step1}, with two passes ($v=2$), the oversampling parameter $P=10$, and the tubal rank $R=15$ to reconstruct the incomplete videos. The PSNR of all reconstructed frames of the ``Akiyo'' video are shown in Figure  \ref{Video_com_1} (upper). Also, the original, the observed, and the reconstruction of some frames are displayed in Figure \ref{Video_com_1} (bottom). The obtained results for the "News" and the "Foreman" videos are reported in Figures \ref{Video_com_2} and \ref{Video_com_3}, respectively. Similar results as the images were achieved for the case of videos and the proposed algorithms provided almost the same recovery performance as the baseline methods by in a faster time. These results are skipped in order to prevent duplication. This clearly indicates the efficiency of the proposed algorithm at delivering satisfactory results faster.
\end{exa}

\begin{figure*}
\begin{center}
\includegraphics[width=.65\columnwidth]{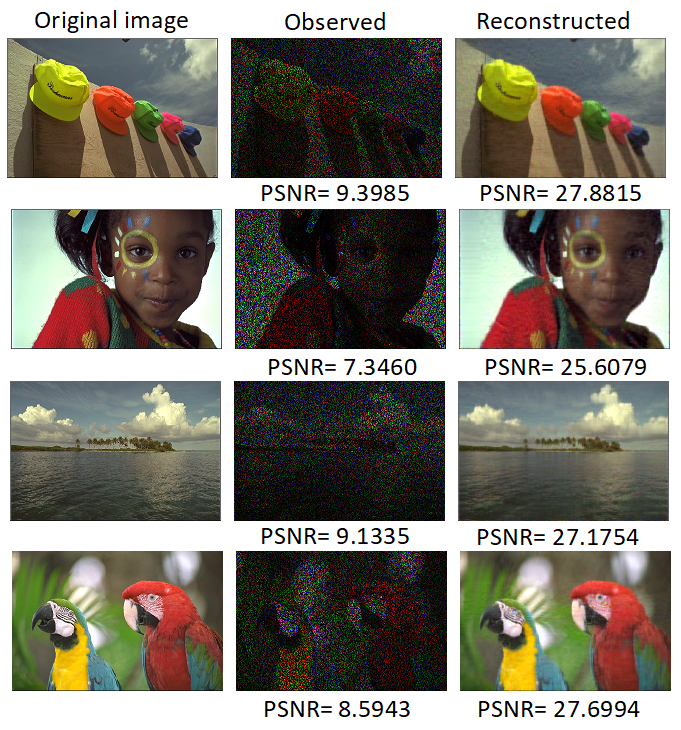}
\caption{\small{Comparing the original, the observed (with $80\%$ missing pixels) and the reconstructed images using the completion algorithm based on the proposed algorithm using two passes and the tubal rank $30$ for Example \ref{salman_com}.}}\label{Im_1}
\end{center}
\end{figure*}

\begin{table}
\begin{center}
\caption{Comparison of the running time and the PSNR achieved by the proposed algorithm, Algorithm \ref{ALG:t-SVD} and Algorithm \ref{ALgRRM} for Example \ref{salman_com}.}\label{Table_Salman_imacom}
\vspace{0.2cm}
{\smaller
\begin{tabular}{||c| c | c ||} 
 \multicolumn{3}{c}{Kodim03}\\
 \hline
Algorithms  & Running Time (Seconds) & PSNR \\
 \hline\hline
 Truncated t-SVD \cite{kilmer2011factorization,kilmer2013third} & 17.34  &  {\bf 28.01 }  \\ 
  Randomized t-SVD \cite{zhang2018randomized} &  10.03 &  27.93  \\ 
 Proposed algorithm  & {\bf 5.1}  &  27.88  \\
 \hline
 \multicolumn{3}{c}{Kodim15}\\
 \hline
 Algorithms  & Running Time (Seconds) & PSNR \\
 \hline\hline
 Truncated t-SVD \cite{kilmer2011factorization,kilmer2013third} &  20.13 &  {\bf 25.98}  \\ 
  Randomized t-SVD \cite{zhang2018randomized} &  9.84 &  25.75 \\ 
 Proposed algorithm  & {\bf 4.57}  &  25.60  \\
 \hline
 \multicolumn{3}{c}{Kodim16}\\
 \hline
 Algorithms  & Running Time (Seconds) & PSNR \\
 \hline\hline
 Truncated t-SVD \cite{kilmer2011factorization,kilmer2013third} & 19.62 &  {\bf 27.56}  \\ 
  Randomized t-SVD \cite{zhang2018randomized} & 11.05  &  27.34  \\ 
 Proposed algorithm  & {\bf 4.89}  &  27.17  \\
 \hline
 \multicolumn{3}{c}{Kodim23}\\
 \hline
 Algorithms  & Running Time (Seconds) & PSNR \\
 \hline\hline
 Truncated t-SVD \cite{kilmer2011factorization,kilmer2013third} & 19.45  &  {\bf 27.89}  \\ 
  Randomized t-SVD \cite{zhang2018randomized} &  11.43 &   27.75 \\ 
 Proposed algorithm  & {\bf 5.21}  &  27.69 \\\hline
\end{tabular}
}
\end{center}
\end{table}

\begin{figure*}
\begin{center}
\includegraphics[width=.75\columnwidth]{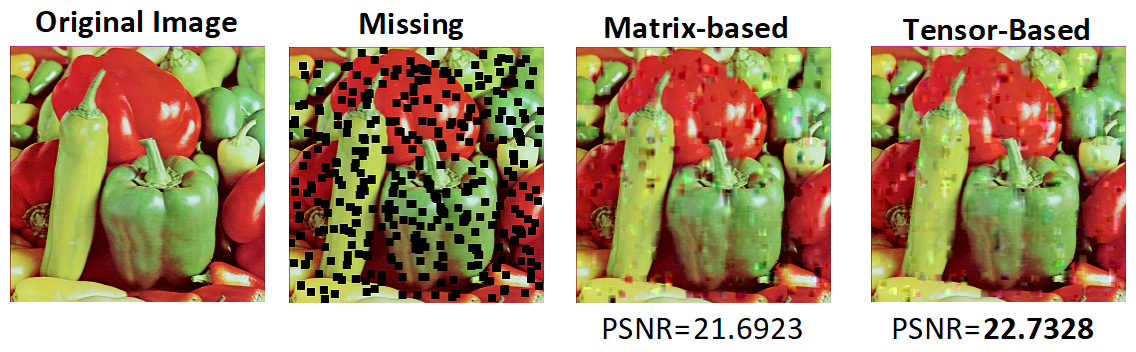}
\caption{\small{Comparing the matrix and tensor based completion algorithms. For the matrix based method the matrix rank $25$ and for the tensor based approach the tubal rank $25$ were used for Example \ref{salman_com}.}}\label{Im_com_t_m}
\end{center}
\end{figure*}

\begin{figure*}
\begin{center}
\includegraphics[width=.9\columnwidth]{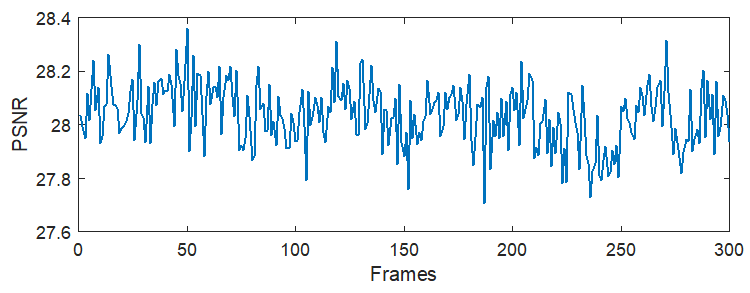}\\\includegraphics[width=.9\columnwidth]{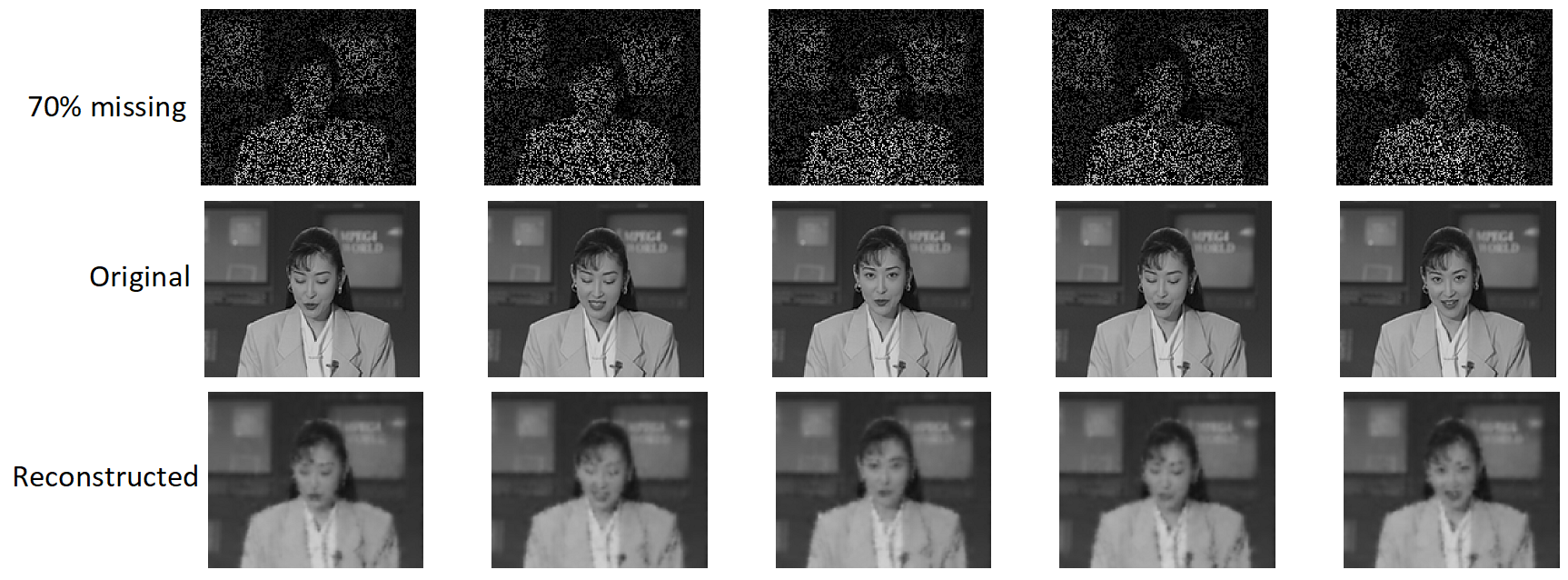}
\caption{\small{(Upper) The PSNR of all reconstructed frames of the ``Akiyo'' video using the
completion procedure combined by the proposed algorithm. The tubal rank $R=15$ and two passes (Bottom) Visualization of some random samples of the original, the observed ($70\%$ missing pixels) and the  reconstructed frames for Example \ref{salman_com}.}}\label{Video_com_1}
\end{center}
\end{figure*}

\begin{figure*}
\begin{center}
\includegraphics[width=.9\columnwidth]{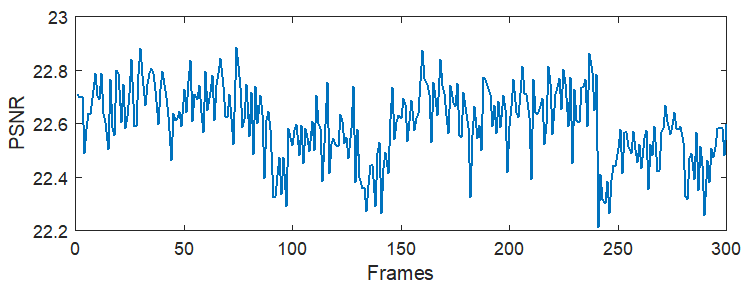}\\\includegraphics[width=.9\columnwidth]{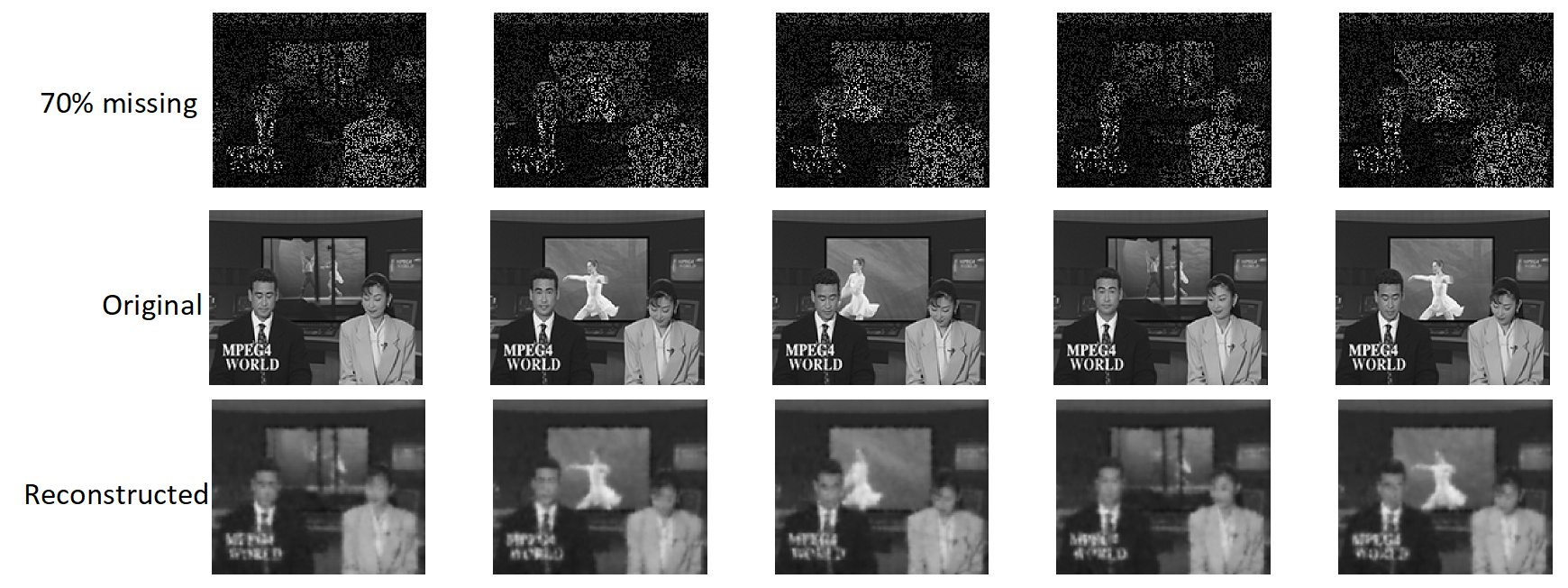}
\caption{\small{(Upper) The PSNR of all reconstructed frames of the ``News'' video using the
completion procedure combined by the proposed algorithm. The tubal rank $R=15$ and two passes (Bottom) Visualization of some random samples of the original, the observed ($70\%$ missing pixels) and the  reconstructed frames for Example \ref{salman_com}.}}\label{Video_com_2}
\end{center}
\end{figure*}

\begin{figure*}
\begin{center}
\includegraphics[width=.9\columnwidth]{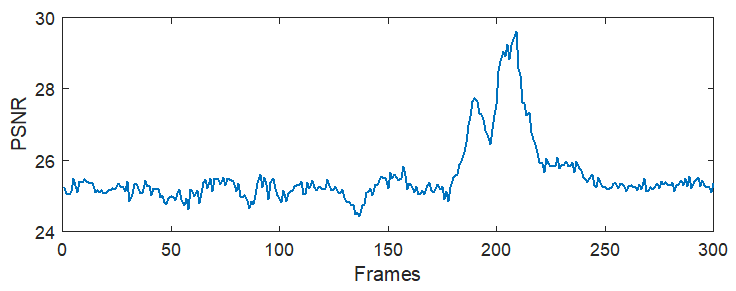}\\\includegraphics[width=.9\columnwidth]{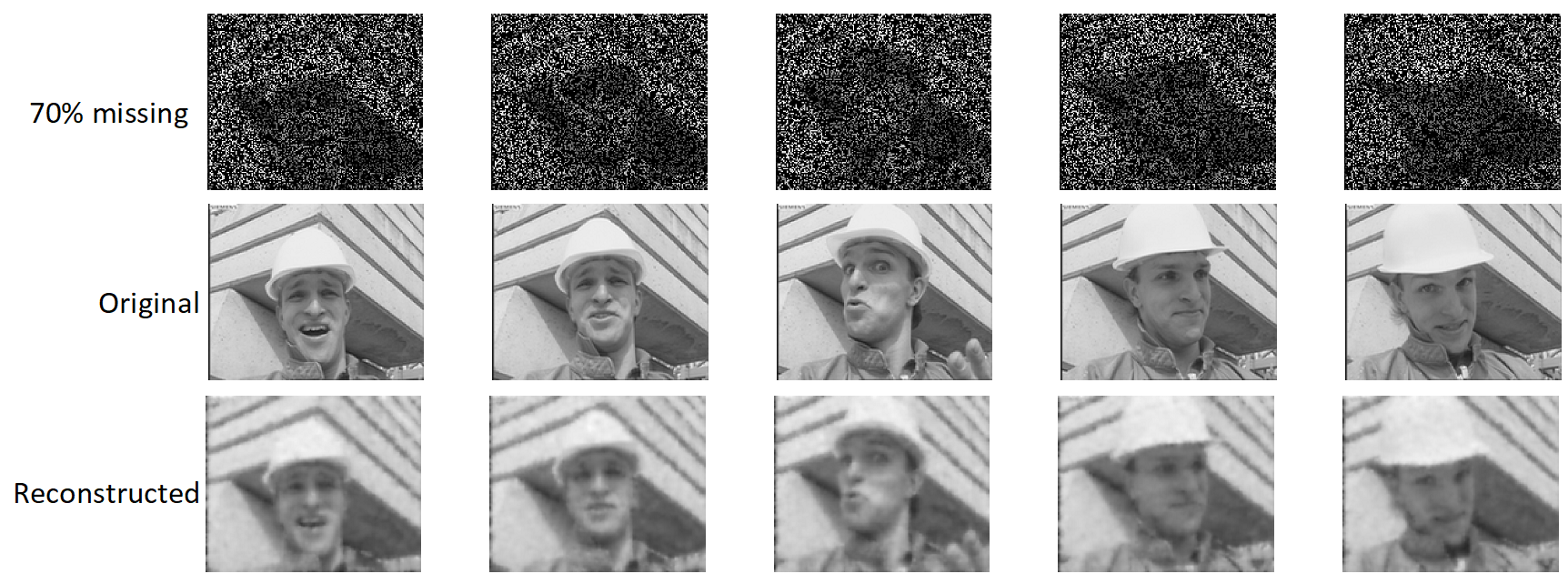}
\caption{\small{(Upper) The PSNR of all reconstructed frames of the ``Foreman'' video using the
completion procedure combined by the proposed algorithm. The tubal rank $R=15$ and two passes (Bottom) Visualization of some random samples of the original, the observed ($70\%$ missing pixels) and the reconstructed frames for Example \ref{salman_com}.}}\label{Video_com_3}
\end{center}
\end{figure*}


\section{Conclusion and future works}\label{Sec:Conclu}
In this paper, a pass-efficient randomized algorithm was proposed for the computation of the tensor SVD (t-SVD). Contrary to the classical randomized subspace algorithms, which need an even number of passes over the data tensor, it can find a low tubal rank approximation of a third-order tensor using an arbitrary number of passes. We applied the proposed algorithm for developing a fast completion method to reconstruct images and videos with missing pixels. The suggested approach is for real-valued tensors of third order, but it may be easily extended to higher-order complex tensors. The thorough simulation results demonstrated the feasibility and efficiency of the proposed algorithm. In future work, we plan to extend the block Krylov subspace algorithms to the tensor case based on the t-product and also propose efficient single-pass algorithms for the computation of the t-SVD.

\section{Acknowledgement} The authors would like to thank the editor and two reviewers for their constructive comments, which have greatly improved the quality of the paper. The
work was partially supported by the Ministry of Education and Science (grant 075.10.2021.068).

\section{Conflict of Interest Statement}
The author declares that he has no
conflict of interest with anything.

\section{Data Availability}
Data openly available in public repositories. The Kodak dataset is accessible at \url{https://www.kaggle.com/datasets/sherylmehta/kodak-dataset}

\section{Appendix} 
 The proof of Theorem \ref{thm_2} is almost the same as the proof of Theorem \ref{THMM_1} presented in \cite{zhang2018randomized}. In fact, they worked on the matrix $\left({\bf X}{\bf X}^T\right)^v{\bf X}$ and here we consider ${\bf Y}=\left({\bf X}^T{\bf X}\right)^v$. However, due to some tricky modifications, we provide the details.  

Let ${\bf X}=[{\bf U}_1,{\bf U}_2]\begin{bmatrix}
{\bf \Sigma}_1& {\bf 0}\\
{\bf 0}&{\bf \Sigma}_2
\end{bmatrix}\begin{bmatrix}
{\bf V}_1^T\\{\bf V}_2^T
\end{bmatrix}$
and define ${\bf \Omega}_1={\bf V}^T_1{\bf \Omega},\,\,{\bf \Omega}_2={\bf V}^T_2{\bf \Omega},$ where ${\bf \Sigma}_1$ and ${\bf \Sigma}_2$ are square. Now, from straightforward computations we have 
\begin{eqnarray*}
{\bf Y}=({\bf X}^T{\bf X})^v{\bf \Omega}=\begin{bmatrix}
{\bf V}_1&{\bf V}_2
\end{bmatrix}\begin{bmatrix}
{\bf \Sigma}^{2v}_1&{\bf 0} \\
{\bf 0}&{\bf \Sigma}^{2v}_2
\end{bmatrix}\begin{bmatrix}
{\bf V}^T_1\\{\bf V}^T_2
\end{bmatrix}
{\bf \Omega}=\\
\begin{bmatrix}
{\bf V}_1&{\bf V}_2
\end{bmatrix}
\begin{bmatrix}
{\bf \Sigma}^{2v}_1& {\bf 0}\\
{\bf 0}&{\bf \Sigma}^{2v}_2
\end{bmatrix}\begin{bmatrix}
{\bf \Omega}_1\\{\bf \Omega}_2
\end{bmatrix},
\end{eqnarray*}
and consider the orthogonal projector ${\bf P}_{\bf V}={\bf V}{\bf V}^T$.  

The next propositions are used in our subsequent analysis. 

\begin{pre} \cite{halko2011finding}\label{pre_1}
Suppose {\bf U} is unitary. Then ${\bf U}^T{\bf P}_{\bf M}{\bf U} = {\bf P}_{{\bf U}^T{\bf M}}$.
\end{pre}

\begin{pre}\cite{halko2011finding}\label{pre_2}
Suppose ${\rm range}({\bf N}) \subset {\rm range}({\bf M})$. Then, for each matrix ${\bf A}$, it
holds that $\|{\bf P}_{\bf N} {\bf A}\|\leq\|{\bf P}_{\bf M}{\bf A}\|$ and that 
$\|({\bf I}-{\bf P}_{\bf M}){\bf A}\|\leq\|({\bf I}-{\bf P}_{\bf N}){\bf A}\|$.
\end{pre}

To prove Theorem \ref{thm_2}, we need to first prove the following theorem.

\begin{thm}\label{imtheo}
Let ${\bf X}\in\mathbb{R}^{I\times J}$ with the SVD $\,\,{\bf X}={\bf U}{\bf \Sigma}{\bf V}^T$. Choose ${\bf \Omega}_1$ and ${\bf \Omega}_2$ as above and assume that ${\bf \Omega}_1$ is of full rank. Compute the ${\bf Q}$ using Algorithm \ref{ALg_2} for an odd number of passes $v$, then the approximation error satisfies 
\begin{eqnarray}\label{EQprinc}
\|{\bf X}({\bf I}-{\bf Q}{\bf Q}^T)\|_F^2\leq\|{\bf \Sigma}_2\|_F^2+{\tau}^{2(2v-1)}\|{{\bf\Sigma}_2}{\bf \Omega}_2{\bf \Omega}_1^{\dag}\|_F^2.
\end{eqnarray}
\end{thm}

\begin{proof}
Define matrices ${\bf Z}$ and ${\bf F}$ as 
\begin{eqnarray}
{\bf Z}={\bf V}^T{\bf Y}{\bf \Omega}_1^{\dag}{\bf \Sigma}_1^{-2v}=\begin{bmatrix}
{\bf I}\\{\bf F}
\end{bmatrix},\quad
{\bf F}\equiv{\bf \Sigma}_2^{2v}{\bf \Omega}_2{\bf \Omega}_1^{\dag}{\bf \Sigma}^{-2v}_1.
\end{eqnarray}
From the construction of ${\bf Z}$, we have
\begin{equation}\label{Ident}
{\rm Range}({\bf Z})\subset{\rm Range}({\bf Y})={\rm Range}({\bf V}^T{\bf Y})={\rm Range}({\bf V}^T{\bf Q}).
\end{equation}
It is not difficult to see 
\begin{eqnarray}\label{EQQ1}
\nonumber
\|{\bf X}({\bf I}-{\bf Q}{\bf Q}^T)\|_F^2=\|{\bf U}{\bf \Sigma}{\bf V}^T({\bf I}-{\bf P}_{\bf Q})\|_F^2=\\
\|{\bf \Sigma}{\bf V}^T({\bf I}-{\bf P}_{\bf Q})\|_F^2,
\end{eqnarray}
and from \eqref{Ident}-\ref{EQQ1} together with Propositions \ref{pre_1} and \ref{pre_2} , we have
\begin{eqnarray}
\nonumber
\|{\bf \Sigma}{\bf V}^T({\bf I}-{\bf P}_{\bf Q})\|_F^2=\|{\bf \Sigma}{\bf V}^T({\bf I}-{\bf P}_{\bf Q}){\bf V}{\bf \Sigma}^T\|_F\leq\\
\|{\bf \Sigma}({\bf I}-{\bf P}_{{\bf V}^T\bf Q}){\bf \Sigma}^T\|_F= 
\|({\bf I}-{\bf P}_{\bf Z}){\bf \Sigma}^T\|^2_F.
\end{eqnarray}
Similar to \cite{halko2011finding,zhang2018randomized}, we can prove 
\[
({\bf I}-{\bf P}_{\bf Z}){\bf \Sigma}^T=\begin{bmatrix}
({\bf I}+{\bf F}^T{\bf F})^{-1}{\bf F}^T{\bf F}{\bf \Sigma}_1\\
({\bf I}-{\bf F}({\bf I}+{\bf F}^T{\bf F})^{-1}){\bf F}^T{\bf \Sigma}_2
\end{bmatrix},
\]
and so we come at 
\begin{eqnarray}\label{Tterm}
\nonumber
\|({\bf I}-{\bf P}_{\bf Z}){\bf \Sigma}^T\|^2_F=\|({\bf I}+{\bf F}^T{\bf F})^{-1}{\bf F}^T{\bf F}{\bf \Sigma}_1\|_F^2+\\ \|({\bf I}-{\bf F}({\bf I}+{\bf F}^T{\bf F})^{-1}){\bf F}^T{\bf \Sigma}_2\|_F^2.
\end{eqnarray}
Now, we bound two terms of \eqref{Tterm}. For the first term, consider 
\begin{eqnarray}\label{EQQ3}
\nonumber
\|({\bf I}+{\bf F}^T{\bf F})^{-1}{\bf F}^T{\bf F}{\bf \Sigma}_1\|_F^2\leq\|{\bf F}({\bf I}+{\bf F}^T{\bf F})^{-1}\|_2\|{\bf F}{\bf \Sigma}_1\|_F\\\leq\|{\bf F}{\bf \Sigma}_1\|_F,
\end{eqnarray}
and for the second term, we get
\begin{eqnarray}\label{EQQ4}
\|({\bf I}-{\bf F}({\bf I}+{\bf F}^T{\bf F})^{-1}){\bf F}^T{\bf \Sigma}_2\|_F^2\leq\|{\bf \Sigma}_2\|^2_F,
\end{eqnarray}
because ${\bf I}-{\bf F}({\bf I}+{\bf F}^T{\bf F})^{-1}){\bf F}^T\preceq{\bf I}$, (see \cite{halko2011finding} for the proof).
From \eqref{EQQ1}, \eqref{Tterm}, \eqref{EQQ3} and \eqref{EQQ4}, we have
\begin{eqnarray}\label{eq_}
\|{\bf X}({\bf I}-{\bf P}_{\bf Q})\|^2_F\leq\|{\bf \Sigma}({\bf I}-{\bf P}_{\bf Z})\|^2_F\leq\|{\bf \Sigma}_2\|_F^2+\|{\bf F}{\bf \Sigma}_1\|_F^2.
\end{eqnarray}
It is seen that ${\bf F}{\bf \Sigma}_1={\bf\Sigma}_2^{2v-1}({\bf\Sigma}_2{\bf\Omega}_2{\bf\Omega}_1^{\dag}){\bf\Sigma}_1^{-2v+1}$ and with some straightforward computations we get
\begin{eqnarray}\label{EQQ5}
\nonumber
\|{\bf F}{\bf \Sigma}_1\|_F\leq\|{\bf\Sigma}_2^{2v-1}\|_2\|{\bf\Sigma}_1^{-(2v-1)}\|_2\|{\bf\Sigma}_2{\bf\Omega}_2{\bf\Omega}_1^{\dag}\|_F\leq\\
{\tau}^{(2v-1)}\|{\bf\Sigma}_2{\bf\Omega}_2{\bf\Omega}_1^{\dag}\|_F.
\end{eqnarray}
From \eqref{EQQ4}- \eqref{EQQ5}, we can conclude the identity \eqref{EQprinc}.
\end{proof}

{\bf Proof of Theorem \ref{thm_2}.} Combining Theorem \ref{imtheo} and Theorem \ref{halko}, the desired result can be achieved.

\begin{thm}\label{halko}
 (average Frobenius error) \cite{halko2011finding}. Suppose that A is a real $m \times n$
matrix with singular values $\sigma_1 \geq \sigma_2 \geq \sigma_3 \geq \cdots$. Choose a target rank $k \geq 2$ and
an oversampling parameter $p \geq 2$, where $k + p \geq \min\{m, n\}$. Draw an $n \times (k + p)$
standard Gaussian matrix ${\bf \Omega}$, and construct the sample matrix ${\bf Y} = {\bf A}{\bf \Omega}$. Then the
expected approximation error
$\mathbb{E}(\|{\bf I} - {\bf P}_{\bf Y} ){\bf A}\|_F\leq(1+\frac{k}{p-1})^{1/2}(\Sigma_{j>k}\sigma_j^2)^{1/2}$
.
\end{thm}

{\bf Proof of Theorem \ref{thm_4}.} 
From the linearity of the expectation operator and relation \eqref{eq_fou} we have 
\begin{eqnarray}\label{Summ_1}
\nonumber
\mathbb{E}\,\left(\|\cX-\cX*\cQ*\cQ^T\|_F^2\right)\leq\hspace*{2cm}\\
\frac{1}{I_3}\left(\sum_{i=1}^{I_3}\mathbb{E}\,\|\widehat{\bf X}^{(i)}-\widehat{\bf X}^{(i)}\widehat{\bf Q}^{(i)}\widehat{\bf Q}^{(i)\,T}\|_F^2\right),
\end{eqnarray}
where $\widehat{\bf X}^{(i)}=\widehat{\cX}(:,:,i)$ and $\widehat{\bf Q}^{(i)}=\widehat{\cQ}(:,:,i)$.
We can now use Theorem \ref{thm_2} to bound each term of summation \eqref{Summ_1} as follows
\begin{eqnarray*}
\mathbb{E}\left(\|\widehat{\bf X}^{(i)}-\widehat{\bf X}^{(i)}\widehat{\bf Q}^{(i)}\widehat{\bf Q}^{(i)\,T}\|_F^2\right)\leq\hspace*{2cm}\\
\frac{1}{I_3}\left(1+\frac{R}{P-1}(\tau^{(i)}_R)^{2(2q-1)}\right)\left(\sum_{j>R}(\widehat{\sigma}^{(i)}_j)^2\right),
\end{eqnarray*}
and so we have 
\begin{eqnarray*}
\mathbb{E}\left(\|(\cX-\cX*\cQ*\cQ^T)\|_F^2\right)\leq\hspace{3cm}\\
\sum_{i_3=1}^{I_3}
\frac{1}{I_3}\left(1+\frac{R}{P-1}(\tau^{(i)}_R)^{2(2q-1)}\right)\left(\sum_{j>R}(\widehat{\sigma}^{(i)}_j)^2\right).
\end{eqnarray*}
It suffices to use the Holder's identity as follows 
\[
\mathbb{E}\left({\,\|\cX-\cX*\cQ*{\cQ}^T\|}_F\right)\leq\left(\mathbb{E}\left({\,\|\cX-\cX*\cQ*{\cQ}^T\|_F^2}\right)\right)^{1/2},
\]
to get the desired result.

\bibliographystyle{elsarticle-num} 
\bibliography{cas-refs}


\end{document}